\documentclass[11pt]{amsart}
\usepackage{amsmath,amssymb,epsfig,color,mathrsfs}
\usepackage{wrapfig,lipsum,sidecap}
\usepackage{graphicx}
\usepackage{multirow}
\usepackage{hhline}
\usepackage{hyperref}
\usepackage{comment}
\usepackage{enumitem}
\setlist[enumerate]{leftmargin=.5in}
\setlist[itemize]{leftmargin=.5in}

\newtheorem{theorem}{Theorem}[section]

\newtheorem{corollary}[theorem]{Corollary}
\newtheorem{definition}[theorem]{Definition}
\newtheorem{lemma}[theorem]{Lemma}
\newtheorem{proposition}[theorem]{Proposition}
\newtheorem{remark}[theorem]{Remark}
\newtheorem{example}[theorem]{Example}

\newcommand{\vanish}[1]{}
\newcommand{\0}{\widehat{0}}

\newcommand{\ve}{\varepsilon}
\newcommand{\sign}{\operatorname{sign}}

\def\p{\prime}

\numberwithin{equation}{section}

\begin{document}
\title{Generalized formulas for the Jones polynomial of a rational link}
\author{Yuanan Diao and G\'abor Hetyei}
\address{Department of Mathematics and Statistics\\
University of North Carolina Charlotte\\
Charlotte, NC 28223}
\email{}
\date{\today}

\begin{abstract}
We derive new formulas for the Jones polynomial and the Kauffman bracket
polynomial of a rational link represented by a standard diagram that is not necessarily
alternating. These formulas generalize the results of Qazaqzeh,
Yasein, and Abu-Qamar for the Tutte polynomial of the Tait graph of an
alternating diagram of a rational link, as well as the matrix formulas of
Lawrence and Rosenstein for the Jones polynomial of a rational link.
Our approach uses the colored version of Brylawski's
tensor product formula for Tutte polynomials of colored graphs, due to
Diao, Hetyei, and Hinson. Furthermore, generalizing the formulas of Qazaqzeh,
Yasein, and Abu-Qamar, we present a finite automaton that computes the
crossing signs, thereby enabling the calculation of the writhe of a
standard diagram of a rational link.

\end{abstract}  

\maketitle

\vspace{-0.1in}
\section*{Introduction}

The Jones polynomial of an oriented rational link may be computed in several
ways. It may be obtained by substitution into the HOMFLY polynomial
using results on this more general invariant~\cite{DEH,Du,Li-Mi}. One
may also use the skein relations of the Jones polynomial directly. This
method is used in~\cite{Kan} to prove that infinitely many rational
knots have the same Jones polynomial (see also~\cite{Sto}). A third
approach is to compute the Kauffman bracket polynomial~\cite{La-Ro,LZ,LLS,Q},
from which the Jones polynomial may be obtained after computing the
writhe of the oriented rational link diagram. Such computations may be
recursive~\cite{LZ,LLS} or may rely on the fact that the Kauffman bracket
polynomial of an alternating link diagram is closely related to the
Tutte polynomial of its Tait graph~\cite{Q}.

We follow the third route and extend the Tait-graph method of Qazaqzeh,
Yasein, and Abu-Qamar~\cite{Q} from alternating standard diagrams with
an odd number of twist boxes to arbitrary (even non-alternating)
standard diagrams. Our formulas differ even
in the special case covered in~\cite[Theorem 3.4]{Q}. We also use our
approach to generalize the matrix formulas of Lawrence and
Rosenstein~\cite{La-Ro}. Computing the signed Tutte polynomial using
our results requires a summation of $F_n$ terms, where $F_n$ is the
$n$th Fibonacci number. This is also the case when one computes
the HOMFLY polynomial via the classical result of Lickorish and
Millett~\cite{Li-Mi}, which, however, requires 
representing a rational link so that each twist box contains an even
number of crossings. As noted in the concluding remarks, 
a recent translation of the Lickorish--Millett result to alternating
rational link diagrams was published in~\cite{DEH}, but that formula
cannot be generalized to arbitrary non-alternating link diagrams.

The key result enabling the derivation of our formulas is the
generalization of Brylawski's tensor product formula to colored graphs,
established in~\cite{DHH}. A tensor product of a graph $G$ with a pointed
graph $\widehat{N}$ (containing a distinguished edge $e$) is obtained by
replacing each edge of $G$ with a copy of $\widehat{N}-e$ in
such a way that the original edge of $G$ is identified with the deleted
edge $e$. Different edges may be replaced with different pointed graphs
in the colored setting. In our work we apply this construction to the 
case where the substituted graphs are augmented paths and their
duals. The inverse operations
were called series-parallel reductions in the work of
Traldi~\cite{Tr}. Generalizing the observations made in~\cite{Q} from
the non-colored setting, in
Section~\ref{sec:Tait} we show that the Tait graph of any unoriented
rational link diagram in standard form is a colored tensor
product of a core graph with augmented path graphs and their
duals. Section~\ref{sec:path} contains the computation of the relevant
colored Tutte invariants of the augmented path graphs and their duals,
while the colored Tutte polynomials of the core graphs are computed in
Section~\ref{sec:core}. In the important special case where division by
our colored variables is allowed, we present two variants in
Section~\ref{sec:core2}. The first variant connects our colored
Tutte polynomial computations to ordinary continued fractions and
leads to an exact evaluation of the Jones polynomial at $t=-1$ in
Section~\ref{sec:subs}. The second reformulation generalizes the main
formula of Lawrence and Rosenstein~\cite[Theorem 4.2]{La-Ro}. This is a
matrix formula in which each matrix factor is associated with a single
twist box, and the size and twisting sign of that twist box determine
the factor. In our generalization, each matrix factor is associated with
a pointed graph replacing a single edge of the core graph, and the core
graph alone determines the corresponding matrix factor.

Our main formulas for the Kauffman bracket of an unoriented rational
link diagram appear in Section~\ref{sec:Kbracket}. These are
obtained by replacing our colored variables with the appropriate power
of the variable $A=t^{-1/4}$, where $t$ is the variable of the Jones polynomial. When presenting our variants of the
Lawrence--Rosenstein formulas in Section~\ref{sec:LR}, we use
the fact that $d=-A^2-A^{-2}$ is nonzero. This makes the substitution
$t=-1$ into the Jones polynomial a particularly interesting special case, since
this forces $d=0$. The absolute value of the Jones polynomial evaluated
at $t=-1$ is well known. Using our results, in
Section~\ref{sec:subs} we derive a formula for the exact value of this
substitution. Variants of the formulas in~\cite{La-Ro} are derived in
Section~\ref{sec:LR}. In particular, we provide a combinatorial formula
that facilitates the computation of the Jones polynomial.

Our general formulas also apply in the special case where the
rational link diagram is alternating and the ordinary Tutte polynomial
of its (unsigned) Tait graph may be used to compute its Kauffman
bracket. In Section~\ref{sec:Taitalt} we provide direct Tutte polynomial
formulas for this important case. We generalize a variant of the results
stated in~\cite{Q}, and we also obtain a variant of the
Lawrence--Rosenstein formulas in this setting. 

To compute the Jones polynomial from the
Kauffman bracket, we must also determine the writhe. One method is presented in~\cite{Q}. In Section~\ref{sec:writhe} we introduce
another approach, inspired by the finite automaton method used
in~\cite{DEH}. Our approach applies not only to
alternating links: it generalizes the key formulas stated in~\cite{Q}
and provides complete proofs in cases that were only stated
in~\cite{Q} to avoid lengthy but analogous verifications.

In the final Section~\ref{sec:concl}, we outline
a possible alternative approach to computing the Jones polynomial of
an alternating link, based on the results of~\cite{DEH}, and present
directions for future research.

\section{Preliminaries}

\subsection{Rational links}

We will use the notation and terminology of~\cite{DEH}.
An unoriented rational link may be presented by a diagram of the form shown
in Figure~\ref{fig:rationallink}; such a diagram is called a \emph{4-plat}.
See Figure~\ref{oddlinktotait} for a concrete example. An unoriented rational
link can be encoded by a continued fraction
$p/q=[0,a_1,a_2,\ldots,a_n]$ where $a_1\cdots a_n\neq 0$, namely
\[
\frac{p}{q}=\cfrac{1}{a_1+\cfrac{1}{a_{2}+\ddots+\cfrac{1}{a_{n-1}+\cfrac{1}{a_n}}}}.
\]
The integers $|a_1|,\ldots,|a_n|$ record the numbers of consecutive half-turn
twists in the twist boxes $B_1,\ldots,B_n$, following the \emph{twist sign}
convention in Figure~\ref{fig:rationallink}.

\begin{figure}[h]
\begin{center}
\scalebox{0.7}{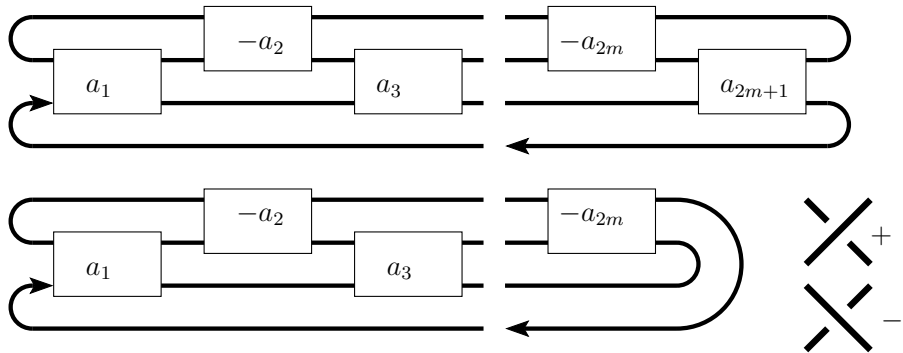}
\end{center}
\caption{The twist sign convention used to define the standard form of an
  unoriented rational link; the role of the arrow will be explained in
  Section~\ref{sec:writhe}.}
\label{fig:rationallink}
\end{figure}

It is an immediate consequence of the definitions that the twist sign
$\sign_t(a_i)$ of each of the $|a_i|$ crossings in twist box $B_i$ is
\begin{equation}
\label{eq:tsign}
\sign_t(a_i)=(-1)^{i-1}\sign(a_i).
\end{equation}
We call such a diagram a \emph{standard diagram} of the unoriented rational
link. The left end of the diagram is fixed, while the closing on the right
depends on the parity of $n$, as indicated in Figure~\ref{fig:rationallink}.
A rational link diagram in standard form is not necessarily alternating.
Indeed, $p/q=[0,a_1,a_2,\ldots,a_n]$ encodes an alternating link diagram if and
only if the continued fraction expansion is in \emph{non-alternating
denominator form}, i.e., all $a_i$ have the same sign. In this special case,
\eqref{eq:tsign} simplifies to
\begin{equation}
\label{eq:tsignalt}
\sign_t(a_i)=(-1)^{i-1}\sign_t(a_1).
\end{equation}
It is known that every rational link admits an alternating diagram;
see~\cite{BZ,Cromwell}.

\begin{remark}{\em
In some literature, the rational numbers used to classify rational links are
the reciprocals of those used in this paper. Here we follow the convention of Duzhin and
Shkolnikov~\cite{Du}.
}
\end{remark}

The Jones polynomial is an invariant of \emph{oriented} links. When we impose an orientation on the diagram, the crossings will acquire a
\emph{crossing sign} as in the convention shown on the right-hand side of
Figure~\ref{channel} in Subsection~\ref{sec:TJ}. If the rational link has one
component, we can choose its orientation arbitrarily since it is known that rational links are invertible. If the rational link has two components, then the top-left
strand and the bottom strand belong to different components, so there are two
choices for the orientation of the top-left strand: left-to-right or
right-to-left. These two choices typically yield two distinct oriented
rational links. 

\vanish{
For ease of reference, we record the following theorems of
Schubert~\cite{Schubert}, which classify unoriented and oriented rational links
in terms of their associated rational numbers.

\begin{theorem}{\em \cite{Schubert}}\label{SchubertTheorem1}
Let $\mathcal{L}(p/q)$ and $\mathcal{L}(p^\p/q^\p)$ be two unoriented rational
links defined by rational numbers $p/q$ and $p^\p/q^\p$. Then they are
topologically equivalent if and only if (i) $q=q^\p$ and (ii) either
$p\equiv p^\p \pmod{q}$ or $p\cdot p^\p\equiv 1 \pmod{q}$.
\end{theorem}

\begin{theorem}{\em \cite{Schubert}}\label{SchubertTheorem2}
Let $\mathcal{L}(p/q)$ and $\mathcal{L}(p^\p/q^\p)$ be two (orientation
compatible) oriented rational links defined by rational numbers $p/q$ and
$p^\p/q^\p$ such that $p$ and $p^\p$ are both odd. Then they are
topologically equivalent as oriented links if and only if (i) $q=q^\p$ and
(ii) either $p\equiv p^\p \pmod{2q}$ or $p\cdot p^\p\equiv 1 \pmod{2q}$.
\end{theorem}

We turn an unoriented link diagram in standard form into an oriented link
diagram in {\em preferred standard form} as follows:
\begin{enumerate}
\item We orient the lowest strand in the link diagram right to left
  (indicated with a solid arrowhead in Figure~\ref{fig:rationallink}).
\item If the link has a second component, we orient the strand in the
  top left corner left to right (indicated  with a hollow arrowhead in
  Figure~\ref{fig:rationallink}). 
\end{enumerate} 
It is well known that a rational link diagram has at most two 
components. }

\subsection{Colored Tutte polynomials}

A classification of Tutte invariants of colored graphs was first
developed by Zaslavsky~\cite{Z} and later generalized by Bollob\'as and
Riordan~\cite{BR}. A unified treatment of these approaches appears in
the work of Ellis-Monaghan and Traldi~\cite{ET}. 
Throughout this paper we follow the notation and terminology of~\cite{BR}.

Let $G$ be a connected graph with $n$ edges, bijectively labeled by $\{1,2,\ldots,n\}$. Assume that each edge is
assigned a color $\lambda$ from a color set $\Lambda$.
For each $\lambda\in\Lambda$ we introduce variables
$X_\lambda,\; x_\lambda,\; Y_\lambda,\; y_\lambda$
and consider the polynomial ring
\[
\mathbb Z[\Lambda]
=
\mathbb Z\!\left[
X_\lambda, x_\lambda, Y_\lambda, y_\lambda
\;:\;
\lambda\in\Lambda
\right].
\]

Let $T$ be a spanning tree of $G$.
An edge $e\in T$ is called \emph{internally active} if for every
edge $f\neq e$ in $G$ such that $(T-e)\cup f$ is a spanning tree,
the label of $e$ is smaller than the label of $f$.
Otherwise $e$ is \emph{internally inactive}.
We assign weight $X_\lambda$ (respectively $x_\lambda$) to each
internally active (respectively inactive) edge of color $\lambda$.
An edge $f\in G-T$ is \emph{externally active} if $f$ has the smallest
label among the edges in the unique cycle contained in $T\cup f$.
Otherwise $f$ is \emph{externally inactive}.
We assign weight $Y_\lambda$ (respectively $y_\lambda$) to each
externally active (respectively inactive) edge of color $\lambda$.

The weight $w(T)$ of a spanning tree $T$ is the product of the
weights of its edges.
The \emph{colored Tutte polynomial} $T(G)$ is defined as
\[
T(G)=\sum_{T} w(T),
\]
where the sum is taken over all spanning trees of $G$.

\begin{theorem}[Bollob\'as--Riordan]
\label{T_BR}
Let $I$ be an ideal of $\mathbb Z[\Lambda]$.
The image of $T(G)$ in $\mathbb Z[\Lambda]/I$ is independent of the
edge labeling of $G$ if and only if, for all
$\lambda,\mu,\nu\in\Lambda$, the differences
\[
\det\!\begin{bmatrix}
X_\lambda & y_\lambda\\
X_\mu & y_\mu
\end{bmatrix}
-
\det\!\begin{bmatrix}
x_\lambda & Y_\lambda\\
x_\mu & Y_\mu
\end{bmatrix},
\]
\[
Y_\nu
\det\!\begin{bmatrix}
x_\lambda & Y_\lambda\\
x_\mu & Y_\mu
\end{bmatrix}
-
Y_\nu
\det\!\begin{bmatrix}
x_\lambda & y_\lambda\\
x_\mu & y_\mu
\end{bmatrix},
\]
and
\[
X_\nu
\det\!\begin{bmatrix}
x_\lambda & Y_\lambda\\
x_\mu & Y_\mu
\end{bmatrix}
-
X_\nu
\det\!\begin{bmatrix}
x_\lambda & y_\lambda\\
x_\mu & y_\mu
\end{bmatrix}
\]
belong to $I$.
\end{theorem}

As shown in~\cite[Lemma 5]{BR}, if $I$ satisfies the conditions of
Theorem~\ref{T_BR}, then $T(G)$ may be computed in
$\mathbb Z[\Lambda]/I$ using the following deletion--contraction recurrence.
Let $e$ be an edge of $G$ of color $\lambda$. Then
\begin{equation}
\label{eq:TGrec}
T(G)=
\begin{cases}
X_\lambda\, T(G/e) & \text{if $e$ is a bridge},\\[4pt]
Y_\lambda\, T(G-e) & \text{if $e$ is a loop},\\[4pt]
x_\lambda\, T(G/e)+y_\lambda\, T(G-e)
& \text{otherwise}.
\end{cases}
\end{equation}

Following~\cite{DHH2}, we regard the colored Tutte polynomial as an
element of $\mathbb Z[\Lambda]/P$, where $P$ is a prime ideal
containing the ideal $I_1$ generated by the polynomials
\[
\det\!\begin{bmatrix}
X_\lambda & y_\lambda\\
X_\mu & y_\mu
\end{bmatrix}
-
\det\!\begin{bmatrix}
x_\lambda & y_\lambda\\
x_\mu & y_\mu
\end{bmatrix}
\ {\rm{and}}\ 
\det\!\begin{bmatrix}
x_\lambda & y_\lambda\\
x_\mu & y_\mu
\end{bmatrix}
-
\det\!\begin{bmatrix}
x_\lambda & Y_\lambda\\
x_\mu & Y_\mu
\end{bmatrix}.
\]

As observed in~\cite{DHH2}, $I_1$ is a prime homogeneous ideal
containing the ideal $I_0$ generated by the differences in
Theorem~\ref{T_BR}.
Since $I_1$ is generated by homogeneous polynomials of degree~$2$,
the variables $X_\lambda,x_\lambda,Y_\lambda,y_\lambda$
represent nonzero congruence classes modulo $I_1$.
Assuming the same holds for $P$, then in the field of fractions
$K$ of $\mathbb Z[\Lambda]/P$, the relations 
\[
\det\!\begin{bmatrix}
X_\lambda & y_\lambda\\
X_\mu & y_\mu
\end{bmatrix}
-
\det\!\begin{bmatrix}
x_\lambda & y_\lambda\\
x_\mu & y_\mu
\end{bmatrix}=0,
\ {\rm{and}}\ 
\det\!\begin{bmatrix}
x_\lambda & y_\lambda\\
x_\mu & y_\mu
\end{bmatrix}
-
\det\!\begin{bmatrix}
x_\lambda & Y_\lambda\\
x_\mu & Y_\mu
\end{bmatrix}=0
\]
are equivalent to
\begin{align}
\frac{X_\lambda-x_\lambda}{y_\lambda}
&=
\frac{X_\mu-x_\mu}{y_\mu}
\label{eq:yconds}\ {\rm{and}}\\
\frac{Y_\lambda-y_\lambda}{x_\lambda}
&=
\frac{Y_\mu-y_\mu}{x_\mu},
\label{eq:xconds}
\end{align}
for all $\lambda,\mu\in\Lambda$.
Hence there exist $u,v\in K$ such that
\begin{align}
\frac{X_\lambda-x_\lambda}{y_\lambda}&=u,
\label{eq:ycondsu}\\
\frac{Y_\lambda-y_\lambda}{x_\lambda}&=v,
\label{eq:xcondsv}
\end{align}
for all $\lambda\in\Lambda$.
Equivalently,
\begin{align}
X_\lambda&=x_\lambda+y_\lambda u,
\label{eq:ycondu}\\
Y_\lambda&=x_\lambda v+y_\lambda.
\label{eq:xcondv}
\end{align}

Thus the parameters $X_\lambda$ and $Y_\lambda$
may be expressed in terms of
$x_\lambda$, $y_\lambda$, $u$, and $v$,
reducing the number of independent parameters from
$4|\Lambda|$ to $2|\Lambda|+2$.
In Zaslavsky’s terminology~\cite{Z},
such Tutte polynomials are called \emph{normal functions}.

The main tool in our proofs will be the colored generalization of
Brylawski’s tensor product formula~\cite{DHH2}.

\begin{definition}
A \emph{pointed colored graph} $\widehat N$
is an undirected graph together with a distinguished edge
$e$ that is neither a loop nor a bridge.
All edges of $\widehat N-e$ are colored by elements of $\Lambda$.
\end{definition}

Next we recall the definition of the
colored versions of the {\em pointed Tutte polynomials
  $T_C(\widehat{N},e)$ and $T_L(\widehat{N},e)$}. The
original non-colored version of these polynomials was first
introduced by Brylawski~\cite{B}.

\begin{definition}\label{def:tctl}
Let $\widehat N$ be a pointed colored graph with distinguished edge $e$.
The polynomial $T_L(\widehat N,e)$ is obtained from the usual
colored Tutte polynomial of $\widehat N-e$, except that internally
active edges in cycles closed by $e$ are treated as internally inactive.
Similarly, $T_C(\widehat N,e)$ is obtained from the usual colored
Tutte polynomial of $\widehat N/e$, except that externally active edges
that would close a cycle containing $e$ are treated as externally inactive.
\end{definition}

It is shown in~\cite[Theorem~2]{DHH2} that
$T_C(\widehat N,e)$ and $T_L(\widehat N,e)$ may be equivalently defined by
\begin{align}
x_\lambda\bigl(T(\widehat N/e)-T_C(\widehat N,e)\bigr)
&=(Y_\lambda-y_\lambda)T_L(\widehat N,e),
\label{TLC1}\\
y_\lambda\bigl(T(\widehat N-e)-T_L(\widehat N,e)\bigr)
&=(X_\lambda-x_\lambda)T_C(\widehat N,e).
\label{TLC2}
\end{align}

\begin{definition}
Let $M$ be a colored graph and $\widehat N$ a pointed colored graph
with distinguished edge $e$.
Fix $\lambda\in\Lambda$.
The \emph{$\lambda$-tensor product}
$M\otimes_\lambda \widehat N$
is the colored graph obtained by replacing each $\lambda$-colored edge of $M$ with a copy of $\widehat N-e$ and
identifying the pointed edge $e$ with the replaced edge.
\end{definition}

\begin{theorem}[Diao--Hetyei--Hinson~\cite{DHH2}]
\label{Ttensor}
The polynomial $T(M\otimes_\lambda \widehat N)$
is obtained from $T(M)$ by the substitutions
\[
X_\lambda\mapsto T(\widehat N-e),\quad
x_\lambda\mapsto T_L(\widehat N,e),\quad
Y_\lambda\mapsto T(\widehat N/e),\quad
y_\lambda\mapsto T_C(\widehat N,e),
\]
while leaving all variables of color $\mu\neq\lambda$ unchanged.
\end{theorem}

\subsection{Tait graph and the Jones polynomial}
\label{sec:TJ}

For a given link diagram $D$ with a checkerboard shading, a crossing can be assigned a $+$ or a $-$ sign relative to this shading as shown on the left side of Figure \ref{channel}. We shall call this sign the {\em shading sign} of the crossing as it is relative to the chosen checkerboard shading. The shading sign is not to be confused with the crossing sign with respect to the orientation of the link which is used in the definition of the writhe of $D$ as shown on the right side of Figure \ref{channel}. 

\begin{figure}[!hbt]
\begin{center}
\includegraphics[scale=0.6]{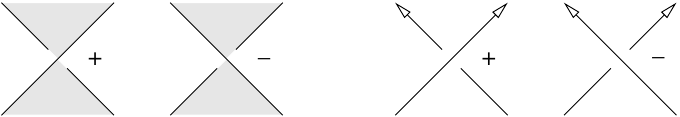}
\end{center}
\caption{Left: the shading sign with respect to the checkerboard shading; Right: the crossing sign with respect to the orientation of the link. 
\label{channel}}
\end{figure}

The Tait graph of a link diagram is obtained from a checkerboard shading of the diagram such that each dark region corresponds to a vertex and a crossing between two adjacent dark regions corresponds to an edge connecting the two corresponding vertices. A Tait graph is a signed graph with the sign of each
edge being the shading sign of the crossing corresponding to the edge. We shall denote the sign of an edge in the Tait graph by $\varepsilon$. We will use
Kauffman's results~\cite{K1,K2} expressing the Jones polynomial of a
link diagram in terms of the signed Tutte polynomial. Note that Kauffman
uses the weights $x_{\varepsilon}$ and $y_{\varepsilon}$ respectively instead
of $X_{\varepsilon}$ and $Y_{\varepsilon}$ respectively; and
he uses $A_{\varepsilon}$ and $B_{\varepsilon}$ respectively, instead 
of $x_{\varepsilon}$ and $y_{\varepsilon}$ respectively. Kauffman's
result may be rephrased as follows.

\begin{theorem}\cite{K2}
\label{T41}
Let $G$ be the signed Tait graph of a regular link projection
$D$ of $K$ as described above, then $T(G)$ equals the Kauffman
bracket polynomial $\langle D\rangle$ under the following variable
substitutions:
$$
X_{\varepsilon}\mapsto -A^{-3\varepsilon},\
Y_{\varepsilon}\mapsto -A^{3\varepsilon},\
x_{\varepsilon}\mapsto A^{\varepsilon},\ y_{\varepsilon}\mapsto
A^{-\varepsilon}.
$$
Furthermore, the Jones polynomial $V_K(t)$ of $K$ can be obtained
from
\begin{equation}\label{eq51}
V_K(t)=(-A^{-3})^{w(D)}\langle D\rangle
\end{equation}
by setting $A=t^{-\frac{1}{4}}$, where $w(D)$ is the writhe of the
projection $D$.
\end{theorem}
The kernel of the map given in Theorem~\ref{T41} is a prime ideal,
  not containing any of the variables
  $X_{\ve},x_{\ve},Y_{\ve},y_{\ve}$. Hence, in the quotient ring, the
  equations~\eqref{eq:ycondu}  and \eqref{eq:xcondv} hold, where both
  $u$ and $v$ equal to
\begin{equation}
\label{eq:d}    
d=\frac{-A^{-3\varepsilon}-A^{\varepsilon}}{A^{-\varepsilon}}=
\frac{-A^{3\varepsilon}-A^{-\varepsilon}}{A^{\varepsilon}}=-A^{-2}-A^2.
\end{equation}
In the special case of an alternating link diagram all edges of the Tait graph
have the same sign, w.l.o.g. we may assume all signs are positive. The
substitution rules stated in Theorem~\ref{T41} may be rewritten as
$$
A^{-1}X_{+}\mapsto -A^{-4},\
AY_{+}\mapsto -A^{4},\
A^{-1}x_{+}\mapsto 1,\ Ay_{+}\mapsto 1.
$$
The number of internal edges (active or inactive) is the same for each
spanning tree, and the same observation holds for the external edges. After
taking out an appropriate factor of $A$ we obtain the following consequence.
\begin{corollary}
\label{cor:T41}  
Let $K$ be an alternating link and $D$ an alternating diagram of $K$.
Let $G$ be the Tait graph of  $D$ as described above, having $v$ vertices
and $e$ edges. Then the Kauffman bracket of $K$ is
$$
\langle D\rangle = A^{2v-e-2} T(G;-A^{-4},-A^{4}).
$$
Here $T(G;X,Y)$ is the Tutte polynomial of the graph $G$.
\end{corollary}

\section{The Tait graph of an unoriented rational link diagram}
\label{sec:Tait}

In this section we describe the structure of the Tait graph associated
to an unoriented rational link diagram. 
For alternating link diagrams, essentially the same description appears
in~\cite{Q}.

\begin{figure}[htbp]
\centering
\includegraphics[scale=0.6]{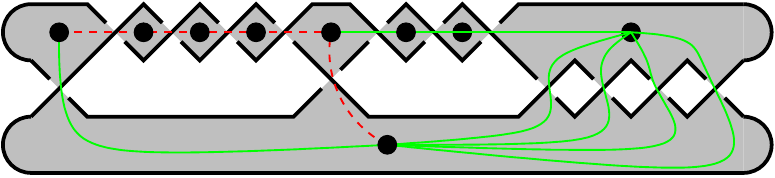}
\caption{An unoriented rational link diagram with an odd number of
twist boxes and its associated Tait graph.
\label{oddlinktotait}}
\end{figure}

\begin{figure}[htbp]
\centering
\includegraphics[scale=0.6]{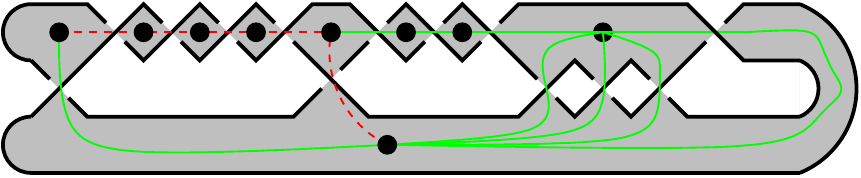}
\caption{An unoriented rational link diagram with an even number of
twist boxes and its associated Tait graph.
\label{evenlinktotait}}
\end{figure}

Consider an unoriented rational link diagram $D$ in standard form,
as in Figures~\ref{oddlinktotait} and~\ref{evenlinktotait}.
(The continued fractions corresponding to these diagrams are
$[0,1,-4,-1,3,4]$ and $[0,1,-4,-1,3,3,1]$, respectively.
Both evaluate to $49/40$, and hence the two diagrams represent the same
unoriented link.)

Choose a checkerboard coloring so that the bounded region
adjacent to the lowest strand is black, and let $G$ be the
corresponding Tait graph.
We call the vertex of $G$ corresponding to this region the
\emph{central vertex}. The remaining vertices of $G$ fall naturally into two classes:
those adjacent to the central vertex, which we call
\emph{major vertices}, and those not adjacent to it, which we call
\emph{minor vertices}. The edges joining a major vertex to the central vertex
correspond precisely to the crossings in the odd-indexed twist boxes.
Each minor vertex has degree two and lies on a path joining two
major vertices. The edges along such a path correspond to the crossings
in an even-indexed twist box. Finally, the shading sign of a crossing in $D$ agrees with its twist sign when the crossing lies in an odd-indexed twist box, and is opposite to
its twist sign when it lies in an even-indexed twist box.
Consequently, all edges incident to the central vertex have the same sign,
and all edges along a path joining two major vertices also have the same sign.

These observations lead to the following lemma.

\begin{lemma}
\label{lem:tsign}
Let $D$ be the standard diagram corresponding to the continued fraction
$[0,a_1,a_2,\ldots,a_n]$, and let $G$ be its Tait graph.
Then the sign of an edge of $G$ corresponding to a crossing in the
$i$-th twist box is equal to $\sign(a_i)$.
\end{lemma}

See Figures~\ref{oddlinktotait} and~\ref{evenlinktotait}
for illustrations of this correspondence.
Lemma~\ref{lem:tsign} immediately implies the following.

\begin{corollary}
\label{cor:alternating}
If $D$ is an alternating rational link diagram, then all edges of its
Tait graph have the same sign $\varepsilon\in\{+,-\}$.
\end{corollary}

These observations motivate the following definition.

\begin{definition}
Let $G$ be the Tait graph of an unoriented rational link diagram in
standard form.
The \emph{underlying core graph} $G_n$ is obtained from $G$ by
replacing each collection of parallel edges between the central vertex
and a major vertex with a single edge, and by replacing each path
connecting two consecutive major vertices with a single edge.
\end{definition}

By construction, the number of edges of $G_n$ equals the number of
twist boxes of $D$.
We label each edge of $G_n$ by the index of the corresponding twist
box. $G_{2n+1}$ and $G_{2n}$ are shown in
Figure~\ref{BaseGraph}.

\begin{figure}[htbp]
\centering
\includegraphics[scale=0.7]{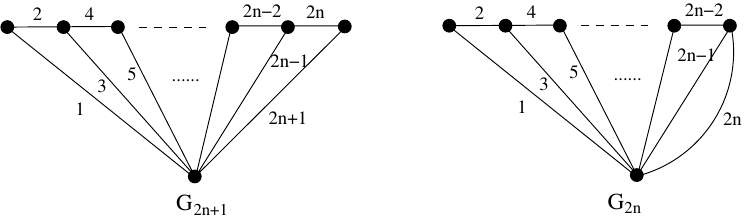}
\caption{The core graphs $G_{2n+1}$ and $G_{2n}$.}
\label{BaseGraph}
\end{figure}

\begin{remark}
Disregarding edge colors, the operations reducing the Tait graph of a rational link diagram to its
core graph are precisely the \emph{series--parallel reductions}
in the sense of Traldi~\cite{Tr}.
\end{remark}

We now observe that the Tait graph of an unoriented rational link
diagram may be reconstructed from its core graph via colored tensor
product operations.

\begin{definition}
The \emph{augmented path graph} $\widehat{P}_m$ is the pointed graph
obtained from a path $S_m$ of length $m$ by joining its endpoints with
a distinguished edge $e$.
Dually, the \emph{augmented dual path graph}
$\widehat{P}_m^*$ consists of $m+1$ parallel edges between two distinct
vertices, one of which is designated as the distinguished edge $e$.
\end{definition}

Equivalently, $\widehat{P}_m$ is a cycle of length $m+1$
with one distinguished edge.
It is planar, and its dual is $\widehat{P}_m^*$.

The following structure theorem is an immediate consequence of the
preceding description.

\begin{theorem}
\label{structure}
Let $G$ be the Tait graph of the unoriented rational link diagram
corresponding to the continued fraction $[0,a_1,a_2,\ldots,a_n]$.
Then $G$ admits the decomposition
\[
G
=
G_n
\otimes_1 \widehat{Q}_1
\otimes_2 \widehat{Q}_2
\cdots
\otimes_n \widehat{Q}_n,
\]
where
\[
\widehat{Q}_i=
\begin{cases}
\widehat{P}_{a_i}, & \text{if $i$ is even},\\[4pt]
\widehat{P}_{a_i}^*, & \text{if $i$ is odd},
\end{cases}
\]
and each non-distinguished edge of $\widehat{Q}_i$ is assigned the
color $\sign(a_i)$.
\end{theorem}

As a consequence of Theorem~\ref{structure}, we may use Theorem~\ref{Ttensor} to compute the Kauffman bracket of a rational link
from any of its standard unoriented diagrams.
It suffices to compute the colored Tutte polynomials of the core graphs
in Figure~\ref{BaseGraph}, together with the pointed signed Tutte
polynomials
\[
T_C(\widehat{P}_{a_i},e),\quad
T_L(\widehat{P}_{a_i},e),\quad
T(\widehat{P}_{a_i}-e),\quad
T(\widehat{P}_{a_i}/e),
\]
and their dual counterparts
\[
T_C(\widehat{P}_{a_i}^*,e),\quad
T_L(\widehat{P}_{a_i}^*,e),\quad
T(\widehat{P}_{a_i}^*-e),\quad
T(\widehat{P}_{a_i}^*/e).
\]

The pointed Tutte polynomials of the signed augmented path graphs and
their duals are computed in Section~\ref{sec:path}, while the signed
Tutte polynomials of the core graphs are computed in
Section~\ref{sec:core}.

\section{The Tutte invariants related to the augmented path graphs and their duals}
\label{sec:path}

In this section we extend the results of Traldi~\cite{Tr}
to colored Tutte polynomials and derive the
corresponding consequences for the Kauffman bracket.

We first compute the (pointed) Tutte polynomials associated to the
augmented path graph $\widehat{P}_m$ directly from the activity-based
definitions.
Label the edges of the path $S_m$ consecutively from $1$ to $m$.
Since $\widehat{P}_m-e$ is the path $S_m$, it has a unique spanning tree,
namely $S_m$ itself.
All its edges are internally active in $T(\widehat{P}_m-e)$
and internally inactive in $T_L(\widehat{P}_m,e)$.
Hence
\begin{equation}
\label{eq:Pm-e}
T(\widehat{P}_m-e)=X_\lambda^m,
\end{equation}
and
\begin{equation}
T_L(\widehat{P}_m,e)=x_\lambda^m.
\end{equation}

The graph $\widehat{P}_m/e$ is a cycle of length $m$.
Its spanning trees are obtained by deleting one edge.
If the deleted edge has label $i\ge2$, then the edges with labels
less than $i$ are internally active, those with larger labels are
internally inactive, and the deleted edge is externally inactive.
If the deleted edge is $e_1$, then all remaining edges are internally
inactive while $e_1$ is externally active.
Therefore
\[
T(\widehat{P}_m/e)
=
\sum_{i=2}^m
X_\lambda^{\,i-1}
y_\lambda
x_\lambda^{\,m-i}
+
Y_\lambda x_\lambda^{\,m-1}.
\]

Using~\eqref{eq:xcondv}, this may be rewritten as
\[
T(\widehat{P}_m/e)
=
\sum_{i=1}^m
X_\lambda^{\,i-1}
y_\lambda
x_\lambda^{\,m-i}
+
x_\lambda^m v,
\]
that is,
\begin{equation}
\label{eq:Pm/e}
T(\widehat{P}_m/e)
=
y_\lambda x_\lambda^{m-1}
[m]_{\frac{X_\lambda}{x_\lambda}}
+
x_\lambda^m v,
\end{equation}
where we use the $q$-notation
\[
[m]_q
=
1+q+q^2+\cdots+q^{m-1}.
\]

The pointed polynomial $T_C(\widehat{P}_m,e)$
is obtained by replacing $Y_\lambda$ with $y_\lambda$ in the preceding
argument. Hence
\begin{equation}
\label{eq:PmTC}
T_C(\widehat{P}_m,e)
=
y_\lambda x_\lambda^{m-1}
[m]_{\frac{X_\lambda}{x_\lambda}}.
\end{equation}

We now compute the (pointed) Tutte polynomials associated to the
pointed graph $N=\widehat{P}^*_m$, which consists of
the distinguished edge $e$ together with $m$ parallel edges.
Then $\widehat{P}_m^*/e$ consists of $m$ loops, and hence
\begin{align}
T(\widehat{P}_m^*/e)
&=
Y_\lambda^m,
\label{eq:Pm*/e}
\\
T_C(\widehat{P}_m^*,e)
&=
y_\lambda^m.
\label{eq:Pm*TC}
\end{align}

To compute $T(\widehat{P}_m^*-e)$ and
$T_L(\widehat{P}_m^*,e)$,
label the $m$ parallel edges of $\widehat{P}_m^*-e$
from $1$ to $m$.
A spanning tree consists of a single edge, say the edge of label $i$;
all edges with larger labels are externally inactive.

If $i\ge2$, then the chosen edge is internally inactive,
and the edges with smaller labels are externally active.
If $i=1$, then that edge is internally active.

Proceeding as in the computation of
$T(\widehat{P}_m/e)$ and using~\eqref{eq:ycondu}, we obtain
\begin{equation}
\label{eq:Pm*-e}
T(\widehat{P}_m^*-e)
=
x_\lambda y_\lambda^{m-1}
[m]_{\frac{Y_\lambda}{y_\lambda}}
+
y_\lambda^m u.
\end{equation}

The computation of $T_L(\widehat{P}_m^*,e)$ is dual to that of
$T_L(\widehat{P}_m,e)$, yielding
\begin{equation}
\label{eq:Pm*TL}
T_L(\widehat{P}_m^*,e)
=
x_\lambda y_\lambda^{m-1}
[m]_{\frac{Y_\lambda}{y_\lambda}}.
\end{equation}

Formulas~\eqref{eq:Pm/e}, \eqref{eq:PmTC}, \eqref{eq:Pm*-e}, and
\eqref{eq:Pm*TL} may be simplified using the following identities.

\begin{lemma}
\label{lem:qxy}
For each $\lambda\in\Lambda$ and integer $m\ge1$, the following equalities hold:
\begin{equation}
\label{eq:qx}
y_\lambda x_\lambda^{m-1}
[m]_{\frac{X_\lambda}{x_\lambda}}
=
\begin{cases}
m\, y_\lambda x_\lambda^{m-1}, & \text{if } u=0,\\[6pt]
\dfrac{X_\lambda^{m}-x_\lambda^{m}}{u}, & \text{if } u\neq0,
\end{cases}
\end{equation}
and
\begin{equation}
\label{eq:qy}
x_\lambda y_\lambda^{m-1}
[m]_{\frac{Y_\lambda}{y_\lambda}}
=
\begin{cases}
m\, x_\lambda y_\lambda^{m-1}, & \text{if } v=0,\\[6pt]
\dfrac{Y_\lambda^{m}-y_\lambda^{m}}{v}, & \text{if } v\neq0.
\end{cases}
\end{equation}
\end{lemma}

\begin{proof}
We prove~\eqref{eq:qx}; the proof of~\eqref{eq:qy} is analogous.

By~\eqref{eq:ycondu} we have $X_\lambda=x_\lambda+y_\lambda u$. If $u=0$, then $\frac{X_\lambda}{x_\lambda}=1$ and
$[m]_1=m$, giving the first case.

Assume now that $u\neq0$.
From the geometric series expansion,
\[
y_\lambda x_\lambda^{m-1}
[m]_{\frac{X_\lambda}{x_\lambda}}
=
y_\lambda x_\lambda^{m-1}
\frac{(\frac{X_\lambda}{x_\lambda})^m-1}{\frac{X_\lambda}{x_\lambda}-1}
=
y_\lambda
\frac{X_\lambda^{m}-x_\lambda^{m}}
{X_\lambda-x_\lambda}.
\]
Substituting
$X_\lambda-x_\lambda$ by $y_\lambda u$ then yields the desired result.
\end{proof}

\section{The colored Tutte polynomial of the core graphs}
\label{sec:core}

In this section we derive closed form formulas for the colored Tutte
polynomials
\[
T_n = T(G_n)
\]
of the core graphs $G_n$.
Theorem~\ref{thm:tilings} provides both our simplest closed form
and the most efficient method for computing these polynomials.

Applying the deletion--contraction recurrence~\eqref{eq:TGrec}
to the edge $e_{2n+1}$ of $G_{2n+1}$, and observing that
$e_{2n}$ becomes a bridge once $e_{2n+1}$ is deleted (and hence must be
contracted), we obtain
\begin{equation}
\label{R_T2n+1}
T_{2n+1}
=
y_{2n+1} X_{2n} T_{2n-1}
+
x_{2n+1} T_{2n}.
\end{equation}

Similarly, applying~\eqref{eq:TGrec} to $e_{2n}$ in $G_{2n}$,
and observing that $e_{2n-1}$ becomes a loop after $e_{2n}$ is
contracted (and hence must be deleted), yields
\begin{equation}
\label{R_T2n}
T_{2n}
=
Y_{2n-1} x_{2n} T_{2n-2}
+
y_{2n} T_{2n-1}.
\end{equation}

For the initial condition we set
\[
T_0 = T(G_0) = 1,
\qquad
T_1 = T(G_1)=T(\{e_1\}) = X_1,
\]
where $G_0$ is the graph consisting of a single vertex.
From the recurrences we obtain
\[
T_2 = Y_1 x_2 + X_1 y_2,
\]
\[
T_3 = X_1 X_2 y_3 + Y_1 x_2 x_3 + X_1 y_2 x_3,
\]
and so forth.

Repeated application of~\eqref{R_T2n+1} and~\eqref{R_T2n}
shows that $T_n$ is a sum of $F_n$ monomials,
where $F_n$ is the number of tilings of a $1\times n$
rectangle by tiles of sizes $1\times1$ and $2\times1$, that is, a Fibonacci number.
This combinatorial interpretation leads to the following
description of $T_n$.

\begin{theorem}
\label{thm:tilings}
For $n\ge1$, the polynomial $T_n$ is the total weight of all tilings
of a $1\times n$ rectangle by tiles of size $1\times1$ and $2\times1$,
with weights assigned as follows:

\begin{itemize}
\item[(i)] A $1\times1$ tile in position $1$ has weight $X_1$;
\item[(ii)] A $1\times1$ tile in any other odd position $2k+1$
has weight $x_{2k+1}$;
\item[(iii)] A $1\times1$ tile in any even position $2k$
has weight $y_{2k}$;
\item[(iv)] A $2\times1$ tile covering positions $(2k-1,2k)$
has weight $Y_{2k-1} x_{2k}$;
\item[(v)] A $2\times1$ tile covering positions $(2k,2k+1)$
has weight $X_{2k} y_{2k+1}$.
\end{itemize}
\end{theorem}

\begin{proof}
The statement follows directly from the recurrences
\eqref{R_T2n+1} and~\eqref{R_T2n} by interpreting
each term of the recurrence as the addition of either a
$1\times1$ tile or a $2\times1$ tile at the end of a tiling.
\end{proof}

\begin{example}{\em 
In the case of $n=4$,  $F_4=5$ and the five possible tilings 
$(1,2,3,4)$,
$(1\text{--}2,3,4)$,
$(1,2\text{--}3,4)$,
$(1,2,3\text{--}4)$ and
$(1\text{--}2,3\text{--}4)$ lead to
\[
T_4
=
X_1 y_2 x_3 y_4
+
Y_1 x_2 x_3 y_4
+
X_1 X_2 y_3 y_4
+
X_1 y_2 Y_3 x_4
+
Y_1 x_2 Y_3 x_4.
\]
}
\end{example}

\begin{example}{\em 
For $n=11$ we have $F_{11}=144$, so the polynomial $T_{11}$ contains
$144$ monomials.
For example, the tiling
$
(1,2,3,4\text{--}5,6\text{--}7,8,9\text{--}10,11)
$
has weight
$
X_1 y_2 x_3 X_4 y_5 X_6 y_7 y_8 Y_9 x_{10} x_{11},
$
and hence contributes this monomial to $T_{11}$.}
\end{example}

We now derive a closed form expression for the sequence
$(T_n)_{n\ge0}$.
From~\eqref{R_T2n+1} and~\eqref{R_T2n}, we may write, for $k\ge2$,
\begin{equation}
\label{eq:Trecgen}
T_k
=
u_k T_{k-2}
+
v_k T_{k-1},
\end{equation}
with initial conditions
\[
T_0=1,
\qquad
T_1=X_1,
\]
where
\begin{equation}
\label{u_n}
u_k=
\begin{cases}
X_{k-1} y_k, & \text{if $k$ is odd},\\[4pt]
Y_{k-1} x_k, & \text{if $k$ is even},
\end{cases}
\end{equation}
and
\begin{equation}
\label{v_n}
v_k=
\begin{cases}
x_k, & \text{if $k$ is odd},\\[4pt]
y_k, & \text{if $k$ is even}.
\end{cases}
\end{equation}

Equation~\eqref{eq:Trecgen} may be rewritten in matrix form as
\begin{equation}
\label{eq:Tmatrix}
\begin{bmatrix}
T_n\\
T_{n-1}
\end{bmatrix}
=
\begin{bmatrix}
v_n & u_n\\
1 & 0
\end{bmatrix}
\begin{bmatrix}
v_{n-1} & u_{n-1}\\
1 & 0
\end{bmatrix}
\cdots
\begin{bmatrix}
v_2 & u_2\\
1 & 0
\end{bmatrix}
\begin{bmatrix}
X_1\\
1
\end{bmatrix}
\end{equation}
for all $n\ge 2$.

In many applications, including the computation of the Kauffman
bracket, the Tutte invariants are considered in the field of fractions
$K$ of $\mathbb Z[\Lambda]/P$, where $P$ is a prime ideal containing
$I_1$ but none of the variables
$X_\lambda,x_\lambda,Y_\lambda,y_\lambda$.
Under this assumption,
equation~\eqref{eq:ycondu} gives
\[
X_1 = x_1 + y_1 u,
\]
and hence
\[
\begin{bmatrix}
X_1\\
1
\end{bmatrix}
=
\begin{bmatrix}
x_1 & y_1\\
1 & 0
\end{bmatrix}
\begin{bmatrix}
1\\
u
\end{bmatrix}.
\]

Setting
\begin{equation}
\label{eq:uv0}
u_1=y_1,
\qquad
v_1=x_1,
\end{equation}
we may rewrite~\eqref{eq:Tmatrix} in the uniform form
\begin{equation}
\label{eq:Tmatrixu}
\begin{bmatrix}
T_n\\
T_{n-1}
\end{bmatrix}
=
\begin{bmatrix}
v_n & u_n\\
1 & 0
\end{bmatrix}
\begin{bmatrix}
v_{n-1} & u_{n-1}\\
1 & 0
\end{bmatrix}
\cdots
\begin{bmatrix}
v_1 & u_1\\
1 & 0
\end{bmatrix}
\begin{bmatrix}
1\\
u
\end{bmatrix},
\end{equation}
 which holds for all $n\ge 1$.
 
\begin{remark}
\label{rem:T-1}
The choices in~\eqref{eq:uv0} are consistent with
\eqref{v_n} and, after formally setting $X_0=1$,
also with~\eqref{u_n}.
With the convention $T_{-1}=u$, the recurrence
\eqref{eq:Trecgen} extends to $k=1$,
yielding
\[
T_1 = y_1 T_{-1} + x_1 T_0=u_1 T_{-1} + v_1 T_0.
\]
\end{remark}

Taking ratios in~\eqref{eq:Tmatrixu} yields the generalized continued
fraction representation
\begin{equation}
\label{eq:Tgencfrac}
\frac{T_n}{T_{n-1}}
=
v_n
+
\cfrac{u_n}{
v_{n-1}
+
\cfrac{u_{n-1}}{
\ddots
+
\cfrac{u_2}{
v_1 + u_1 u
}
}
}.
\end{equation}

\section{Two reformulations of our results on the core graphs}
\label{sec:core2}

In this section we assume that our Tutte invariants lie in the field
of fractions $K$ of $\mathbb Z[\Lambda]/P$, where $P$ is a prime ideal
containing $I_1$ but none of the variables
$X_\lambda,x_\lambda,Y_\lambda,y_\lambda$.
Under this assumption we present two reformulations of the results
of Section~\ref{sec:core}.
The first reformulation converts~\eqref{eq:Tgencfrac}
into an ordinary continued fraction.

\begin{definition}
Let $(p_i)_{i\ge0}$ be a sequence of parameters.
For $m\ge0$, the \emph{generalized semifactorials induced by $(p_i)$}
are defined by
\[
p_{2m}!!
=
\prod_{j=0}^{m} p_{2j},
\qquad
p_{2m+1}!!
=
\prod_{j=0}^{m} p_{2j+1}.
\]
\end{definition}

If $p_0=1$ and $p_i=i$ for $i>0$, this reduces to the usual
definition of semifactorials.

By setting $u_0=1$ and $u_{-1}=1$, we can extend the definition of $u_n!!$ to include the case $n=0$ and $n=-1$, namely $u_0!!=u_{-1}!!=1$. This allows us to define
\begin{equation}
\label{eq:Pndef}
S_n=\dfrac{T_n}{u_n!!}
\end{equation}
for $n\ge -1$.

Using~\eqref{eq:Trecgen} and the definition of $u_0$ and $u_{-1}$, for any $j\ge 0$, we compute
\begin{align*}
S_{2j+1}
&=
\frac{T_{2j+1}}{u_{2j+1}!!}
=
\frac{u_{2j+1}T_{2j-1}+v_{2j+1}T_{2j}}
{u_{2j+1}!!}  \\
&=
S_{2j-1}
+
\frac{v_{2j+1}u_{2j}!!}{u_{2j+1}!!}
S_{2j},
\end{align*}
and
\begin{align*}
S_{2j+2}
&=
\frac{T_{2j+2}}{u_{2j+2}!!}
=
\frac{u_{2j+2}T_{2j}+v_{2j+2}T_{2j+1}}
{u_{2j+2}!!}  \\
&=
S_{2j}
+
\frac{v_{2j+2}u_{2j+1}!!}{u_{2j+2}!!}
S_{2j+1}.
\end{align*}

These relations may be summarized as follows.

\begin{proposition}
\label{prop:Prec}
The sequence $(S_n)$ satisfies
\[
S_n
=
S_{n-2}
+
\frac{v_n\, u_{n-1}!!}{u_n!!}
\, S_{n-1},
\qquad n\ge 1,
\]
with initial conditions
\[
S_{-1}=u,
\qquad
S_0=1.
\]
\end{proposition}

We now obtain the first reformulation of the closed form for $T_n$.

\begin{theorem}
\label{thm:coremain}
Under the assumption that the Tutte invariants lie in the field
of fractions $K$ of $\mathbb Z[\Lambda]/P$, where $P$ is a prime ideal
containing $I_1$ but none of the variables
$X_\lambda,x_\lambda,Y_\lambda,y_\lambda$, we have
\[
T_n = u_n!! \, S_n,
\]
where
\[
\begin{bmatrix}
S_n\\
S_{n-1}
\end{bmatrix}
=
\begin{bmatrix}
\frac{v_n\, u_{n-1}!!}{u_n!!} & 1\\
1 & 0
\end{bmatrix}
\begin{bmatrix}
\frac{v_{n-1}\, u_{n-2}!!}{u_{n-1}!!} & 1\\
1 & 0
\end{bmatrix}
\cdots
\begin{bmatrix}
\frac{v_1u_0!!}{u_1!!} & 1\\
1 & 0
\end{bmatrix}
\begin{bmatrix}
1\\
u
\end{bmatrix},
\qquad n\ge0.
\]
Here $(u_k)$ and $(v_k)$ are given by
\eqref{u_n}, \eqref{v_n}, and~\eqref{eq:uv0}.
\end{theorem}

Taking ratios yields the ordinary continued fraction
\begin{equation}
\label{eq:Pcfrac}
\frac{S_n}{S_{n-1}}
=
\frac{v_n\, u_{n-1}!!}{u_n!!}
+
\cfrac{1}{
\frac{v_{n-1}\, u_{n-2}!!}{u_{n-1}!!}
+
\cfrac{1}{
\ddots
+
\cfrac{1}{
\frac{v_1u_0!!}{u_1!!}
+
\cfrac{1}{u}
}
}
}.
\end{equation}

Our second reformulation is inspired by~\cite[Theorem 4.2]{La-Ro}
and generalizes it to the present setting.
The goal is to replace~\eqref{eq:Tmatrixu} by a matrix product
in which the $k$-th matrix (from the right) involves only variables indexed by $k$.

Extracting the common factor
\[
\prod_{1\le j\le \lfloor n/2\rfloor} X_{2j}
\cdot
\prod_{1\le j\le \lfloor (n+1)/2\rfloor} Y_{2j-1}
\]
from the expression for $T_n$ in Theorem~\ref{thm:tilings}
yields the following reformulation.

\begin{corollary}
\label{cor:tilings}
The polynomial $T_n$ equals
\[
\prod_{1\le j\le \lfloor n/2\rfloor} X_{2j}
\cdot
\prod_{1\le j\le \lfloor (n+1)/2\rfloor} Y_{2j-1}
\]
times the total weight of all tilings of a $1\times n$ rectangle by
$1\times1$ and $2\times1$ tiles, with weights assigned as follows:
\begin{itemize}
\item[(i)] A $1\times1$ tile in position $1$ has weight $\frac{X_1}{Y_1}$;
\item[(ii)] A $1\times1$ tile in any other odd position $2k+1$
has weight $\frac{x_{2k+1}}{Y_{2k+1}}$;
\item[(iii)] A $1\times1$ tile in any even position $2k$
has weight $\frac{y_{2k}}{X_{2k}}$;
\item[(iv)] A $2\times1$ tile covering $(2k-1,2k)$
has weight $\frac{x_{2k}}{X_{2k}}$;
\item[(v)] A $2\times1$ tile covering $(2k,2k+1)$
has weight $\frac{y_{2k+1}}{Y_{2k+1}}$.
\end{itemize}
\end{corollary}

Indeed, the prefactor
\[
\prod_{1\le j\le \lfloor n/2\rfloor} X_{2j}
\cdot
\prod_{1\le j\le \lfloor (n+1)/2\rfloor} Y_{2j-1}
\]
contains exactly one variable indexed by $k$ for each
$1\le k\le n$.
Thus the contribution of each tile in Theorem~\ref{thm:tilings}
may be divided by the variable whose index matches the
position covered by that tile.

In analogy with~\eqref{eq:Tmatrix} and~\eqref{eq:Tmatrixu},
we obtain the following matrix formulation.

\begin{theorem}
\label{thm:lr}
Define
\[
R_n
=
\frac{T_n}{
\prod_{1\le j\le \lfloor n/2\rfloor} X_{2j}
\cdot
\prod_{1\le j\le \lfloor (n+1)/2\rfloor} Y_{2j-1}
}.
\]
Then
\[
R_n
=
\begin{bmatrix}
1 & 0
\end{bmatrix}
M_n M_{n-1}\cdots M_1
\begin{bmatrix}
1\\
u
\end{bmatrix},
\]
where
\[
M_k
=
\begin{cases}
\begin{bmatrix}
\frac{y_k}{X_k} & \frac{x_k}{X_k}\\[6pt]
1 & 0
\end{bmatrix},
& \text{if $k$ is even},\\[14pt]
\begin{bmatrix}
\frac{x_k}{Y_k} & \frac{y_k}{Y_k}\\[6pt]
1 & 0
\end{bmatrix},
& \text{if $k$ is odd}.
\end{cases}
\]
\end{theorem}
In particular, each matrix $M_k$ depends only on the variables indexed by $k$, reflecting the local contribution of the $k$-th twist box.

\section{Formulas for the Kauffman bracket}
\label{sec:Kbracket}

In this section we combine Theorems~\ref{Ttensor} and~\ref{structure}
with the results of Sections~\ref{sec:path}, \ref{sec:core}, and
\ref{sec:core2} to obtain explicit formulas for the Kauffman bracket
of a rational link diagram in standard form.

For brevity, we write
\[
K_n = K\bigl([0,a_1,a_2,\ldots,a_n]\bigr)
\]
for the Kauffman bracket of the unoriented rational link diagram
encoded by the continued fraction $[0,a_1,a_2,\ldots,a_n]$.
Let
\[
\varepsilon_i = \sign(a_i).
\]

Setting $\lambda=\varepsilon\in\{+,-\}$ and applying
Theorem~\ref{T41} to
\eqref{eq:Pm-e}--\eqref{eq:Pm*TL}, we obtain the following
correspondence rules for computing the Kauffman bracket:
\begin{align}
\label{eq:KB1}
T(\widehat{P}_m-e)
&\longmapsto
(-A)^{-3m\varepsilon},
\\
T_L(\widehat{P}_m,e)
&\longmapsto
A^{m\varepsilon},
\\
T(\widehat{P}_m/e)
&\longmapsto
A^{(m-2)\varepsilon}[m]_{-A^{-4\varepsilon}}
+
A^{m\varepsilon} d,
\\
T_C(\widehat{P}_m,e)
&\longmapsto
A^{(m-2)\varepsilon}[m]_{-A^{-4\varepsilon}},
\\
T(\widehat{P}_m^*/e)
&\longmapsto
(-A)^{3m\varepsilon},
\\
T_C(\widehat{P}_m^*,e)
&\longmapsto
A^{-m\varepsilon},
\\
T(\widehat{P}_m^*-e)
&\longmapsto
A^{-(m-2)\varepsilon}[m]_{-A^{4\varepsilon}}
+
A^{-m\varepsilon} d,
\\
T_L(\widehat{P}_m^*,e)
&\longmapsto
A^{-(m-2)\varepsilon}[m]_{-A^{4\varepsilon}}.
\end{align}

Here
$
d=-A^{-2}-A^2
$
as in~\eqref{eq:d}.

Since replacing $\varepsilon$ by $-\varepsilon$ simply interchanges
the two possible edge signs, the two cases in
Lemma~\ref{lem:qxy} lead to the same structural consequence. Thus by Theorem~\ref{T41}, we obtain the following specialization.

\begin{corollary}
\label{cor:qA}
\[
A^{(2-m)\varepsilon}
[m]_{-A^{4\varepsilon}}
=
\begin{cases}
m\, A^{(2-m)\varepsilon},
& \text{if } d=0,\\[6pt]
\dfrac{(-A)^{3m\varepsilon}-A^{-m\varepsilon}}{d},
& \text{otherwise}.
\end{cases}
\]
\end{corollary}

Note that $d=0$ is equivalent to $A^{-4}=-1$.
In the computation of the Jones polynomial, this occurs precisely
when $t=-1$.

For even indices $2j$, we have
\[
\widehat{Q}_{2j}=\widehat{P}_{|a_{2j}|},
\qquad
\varepsilon_{2j}=\sign(a_{2j}),
\]
so the substitutions required to obtain $K_n$ from $T_n$ are
\begin{eqnarray} \label{eq:Xe} X_{2j}&\mapsto &T(\widehat{Q}_{2j}- e)\mapsto (-A)^{-3a_{2j}},\\ \label{eq:Ye} Y_{2j}&\mapsto &T(\widehat{Q}_{2j}/e)\mapsto A^{a_{2j} -2\ve_{2j}}[|a_{2j}|]_{-A^{-4\ve_{2j}}}+ A^{a_{2j}}\cdot d,\\ \label{eq:xe} x_{2j}&\mapsto &T_L(\widehat{Q}_{2j}, e)\mapsto A^{a_{2j}},\\ \label{eq:ye} y_{2j}&\mapsto &T_C(\widehat{Q}_{2j}, e)\mapsto A^{a_{2j} -2\ve_{2j}}[|a_{2j}|]_{-A^{-4\ve_{2j}}}, \end{eqnarray}
(Observe that $Y_{2j}$ does not appear explicitly in the final
expression for $K_n$.)

Similarly, for odd indices $2j+1$,
\[
\widehat{Q}_{2j+1}
=
\widehat{P}_{|a_{2j+1}|}^*,
\qquad
\varepsilon_{2j+1}=\sign(a_{2j+1}),
\]
and we obtain
\begin{align}
\label{eq:Xo}
X_{2j+1}
&\longmapsto
A^{-a_{2j+1}+2\varepsilon_{2j+1}}
[|a_{2j+1}|]_{-A^{4\varepsilon_{2j+1}}}
+
A^{-a_{2j+1}} d,
\\
\label{eq:Yo}
Y_{2j+1}
&\longmapsto
(-A)^{3a_{2j+1}},
\\
\label{eq:xo}
x_{2j+1}
&\longmapsto
A^{-a_{2j+1}+2\varepsilon_{2j+1}}
[|a_{2j+1}|]_{-A^{4\varepsilon_{2j+1}}},
\\
\label{eq:yo}
y_{2j+1}
&\longmapsto
A^{-a_{2j+1}}.
\end{align}

$X_{2j+1}$ appears in the formulation of $K_n$ only for $j=0$.
In that case,
\begin{equation}
\label{eq:X1rule}
X_1
\longmapsto
A^{-a_1+2\varepsilon_1}
[|a_1|]_{-A^{4\varepsilon_1}}
+
A^{-a_1}(-A^2-A^{-2}).
\end{equation}

Applying the substitution rules above to~\eqref{u_n}, we obtain
\begin{align*}
u_{2j+1}
&\longmapsto
X_{2j} y_{2j+1}
\longmapsto
(-A)^{-3a_{2j}} A^{-a_{2j+1}}
=
(-1)^{a_{2j}}
A^{-3a_{2j}-a_{2j+1}},
\\
u_{2j}
&\longmapsto
Y_{2j-1} x_{2j}
\longmapsto
(-A)^{3a_{2j-1}} A^{a_{2j}}
=
(-1)^{a_{2j-1}}
A^{3a_{2j-1}+a_{2j}}.
\end{align*}

These two cases may be summarized as
\begin{equation}
\label{utoA}
u_k
\longmapsto
(-1)^{a_{k-1}}
A^{(-1)^k(3a_{k-1}+a_k)}.
\end{equation}

Similarly, applying the substitution rules to~\eqref{v_n} yields
\begin{equation}
\label{vtoA}
v_k
\longmapsto
A^{(-1)^{k-1}(2\varepsilon_k-a_k)}
[|a_k|]_{-A^{(-1)^{k-1}4\varepsilon_k}}.
\end{equation}

Combining Theorem~\ref{thm:tilings}
with~\eqref{eq:X1rule}, \eqref{utoA}, and~\eqref{vtoA},
we obtain the following explicit description.

\begin{corollary}\label{Cor_K_n}
The Kauffman bracket
\[
K_n=K([0,a_1,a_2,\ldots,a_n])
\]
is the total weight of all tilings of a path of length $n$
by $1\times1$ and $2\times1$ tiles, with weights assigned as follows:
\begin{enumerate}
\item A $1\times1$ tile in position $1$ has weight
\[
A^{-a_1+2\varepsilon_1}
[|a_1|]_{-A^{4\varepsilon_1}}
+
A^{-a_1} d.
\]

\item A $1\times1$ tile in position $k>1$ has weight
\[
A^{(-1)^{k-1}(2\varepsilon_k-a_k)}
[|a_k|]_{-A^{(-1)^{k-1}4\varepsilon_k}}.
\]

\item A $2\times1$ tile covering $(k,k+1)$ has weight
\[
(-1)^{a_k}
A^{(-1)^{k+1}(3a_k+a_{k+1})}.
\]
\end{enumerate}
\end{corollary}

An equivalent expression is obtained by substituting
\eqref{utoA} and~\eqref{vtoA}
directly into the matrix formulation~\eqref{eq:Tmatrixu}.
The verification is straightforward and is omitted.

\begin{example}{\em 
Consider the two-component rational link diagram $D=[0,2,1,3,3]$
with the components oriented so that $w(D)=-5$.
There are five tilings of the $1\times4$ strip: $(1\times1,1\times1,1\times1,1\times1)$, $(1\times1,1\times1,2\times1$, $(1\times1,2\times1,1\times1)$, $(2\times1,1\times1,1\times1$ and $(2\times1,2\times1)$.
Their combined weights are
$-A^{-5}(1+A^8)(1-A^{4}+A^8)(1-A^{-4}+A^{-8})$,
$A^7(1+A^8)$,
$A^{-9}(1+A^{8})(1-A^{-4}+A^{-8})$,
$A^7(1-A^{4}+A^8)(1-A^{-4}+A^{-8})$ and
$-A^{-19}$
respectively.
Summing these contributions,
multiplying by $(-A^3)^{-w(D)}=-A^{15}$,
and substituting $A^{-4}=t$, we obtain
\[
J(D)
=
t^{-\frac12}
\bigl(
- t+2-4t^{-1}+5t^{-2}-6t^{-3}
+6t^{-4}-6t^{-5}+3t^{-6}-2t^{-7}+t^{-8}
\bigr).
\]
}
\end{example}

We next compute the substitutions into $u_k!!$ in order to apply
Theorem~\ref{thm:coremain}.
Recall that $u_1=y_1\mapsto A^{-a_1}$.
Using~\eqref{utoA}, we obtain
\begin{align}
u_{2j+1}!!
&\longmapsto
\prod_{k=0}^{j}
(-1)^{a_{2k}}
A^{-(3a_{2k}+a_{2k+1})}
\nonumber\\
&=
(-A^{-2})^{a_2+a_4+\cdots+a_{2j}}
A^{-(a_1+\cdots+a_{2j+1})},
\label{ueven}
\\[6pt]
u_{2j+2}!!
&\longmapsto
\prod_{k=0}^{j}
(-1)^{a_{2k+1}}
A^{3a_{2k+1}+a_{2k+2}}
\nonumber\\
&=
(-A^{2})^{a_1+a_3+\cdots+a_{2j+1}}
A^{a_1+\cdots+a_{2j+2}}.
\label{uodd}
\end{align}

Taking quotients gives
\[
\frac{u_{2j+1}!!}{u_{2j+2}!!}
\longmapsto
(-A^{-4})^{a_1+\cdots+a_{2j+1}}
A^{-a_{2j+2}},
\]
and, similarly,
\[
\frac{u_{2j}!!}{u_{2j+1}!!}
\longmapsto
(-A^{4})^{a_1+\cdots+a_{2j}}
A^{a_{2j+1}}.
\]

Using~\eqref{vtoA}, we obtain
\begin{align}
\label{eq:vodd}
\frac{v_{2j+2}u_{2j+1}!!}{u_{2j+2}!!}
&\longmapsto
A^{-2\varepsilon_{2j+2}}
[|a_{2j+2}|]_{-A^{-4\varepsilon_{2j+2}}}
(-A^{-4})^{a_1+\cdots+a_{2j+1}},
\\[6pt]
\label{eq:veven}
\frac{v_{2j+1}u_{2j}!!}{u_{2j+1}!!}
&\longmapsto
A^{2\varepsilon_{2j+1}}
[|a_{2j+1}|]_{-A^{4\varepsilon_{2j+1}}}
(-A^{4})^{a_1+\cdots+a_{2j}}.
\end{align}

These may be summarized in the uniform form
\begin{equation}
\label{eq:vrule}
\frac{v_{k+1}u_k!!}{u_{k+1}!!}
\longmapsto
A^{2(-1)^k\varepsilon_{k+1}}
[|a_{k+1}|]_{-A^{4(-1)^k\varepsilon_{k+1}}}
(-A^{4(-1)^k})^{a_1+\cdots+a_k}.
\end{equation}

Similarly,
\eqref{ueven} and \eqref{uodd} may be summarized as
\begin{equation}
\label{eq:urule}
u_{k+1}!!
\longmapsto
(-A^{2(-1)^{k+1}})^{a_k+a_{k-2}+\cdots}
A^{(-1)^{k+1}(a_1+\cdots+a_{k+1})}.
\end{equation}

\begin{definition}
\label{def:go}
Define
\[
\omega_k(a_1,\ldots,a_k)
=
(-A^{2(-1)^k})^{a_{k-1}+a_{k-3}+\cdots}
A^{(-1)^k(a_1+\cdots+a_k)},
\]
and
\[
\gamma_k(a_1,\ldots,a_k)
=
A^{2(-1)^{k-1}\varepsilon_k}
[|a_k|]_{-A^{4(-1)^{k-1}\varepsilon_k}}
(-A^{4(-1)^{k-1}})^{a_1+\cdots+a_{k-1}}.
\]
\end{definition}

Substituting into Theorem~\ref{thm:coremain} yields the following.

\begin{corollary}
\label{cor:Kbfinal}
The Kauffman bracket
\[
K_n=K([0,a_1,\ldots,a_n])
\]
is given by
\[
K_n
=
\omega_n(a_1,\ldots,a_n)\,\kappa_n,
\]
where $\kappa_n$ satisfies
\[
\begin{bmatrix}
\kappa_n\\
\kappa_{n-1}
\end{bmatrix}
=
\begin{bmatrix}
\gamma_n & 1\\
1 & 0
\end{bmatrix}
\cdots
\begin{bmatrix}
\gamma_1 & 1\\
1 & 0
\end{bmatrix}
\begin{bmatrix}
1\\
d
\end{bmatrix}.
\]
\end{corollary}

\begin{remark}\label{rem:gamma}
{\em 
Using \eqref{eq:tsign} we may express $\gamma_k(a_1,\ldots,a_k)$ as $$ \gamma_k(a_1,\ldots,a_k)=A^{2\sign_t(a_k)}[|a_{k}|]_{-A^{4\sign_{t}(a_k)}} (-A^{4(-1)^{k-1}})^{a_1+a_2+\cdots+a_{k-1}}. $$ }
\end{remark}

Finally, substituting~\eqref{eq:Xe}--\eqref{eq:yo}
into Theorem~\ref{thm:lr} gives a fully explicit matrix product.

\begin{theorem}
\label{thm:lrk}
The Kauffman bracket satisfies
\[
K_n
=
(-A)^{3\sum_{k=1}^n (-1)^{k-1}a_k}
\;\rho_n,
\]
where
\[
\rho_n
=
\begin{bmatrix}
1 & 0
\end{bmatrix}
M_n\cdots M_1
\begin{bmatrix}
1\\
d
\end{bmatrix},
\]
with
\begin{equation}
\label{eq:Msubs}
M_k
=
\begin{bmatrix}
- A^{2(-1)^k\varepsilon_k}
[|a_k|]_{-A^{4(-1)^k\varepsilon_k}}
&
(-A^{4(-1)^k})^{a_k}
\\[6pt]
1 & 0
\end{bmatrix}.
\end{equation}
\end{theorem} 

\section{Substitutions into the Kauffman bracket and the Jones polynomial}
\label{sec:subs}

Corollary~\ref{cor:qA} shows that the formulas of
Section~\ref{sec:Kbracket} simplify considerably once a specific
substitution is fixed and we determine whether $d=0$ holds.
We treat the two possible cases separately.

We first consider the special case $d=0$.
When evaluating the Jones polynomial, this corresponds to the substitution
$t=-1$.
In this situation it is convenient to assume that the rational link is
represented by an alternating diagram.

\begin{theorem}
\label{thm:t=-1}
Let $D$ be a rational link diagram encoded by the continued fraction
\[
[0,a_1,a_2,\ldots,a_n]=\frac{p}{q},
\]
where $a_k>0$ for all $k$, and $p,q>0$ are relatively prime.
Then
\[
V_L(-1)
=
(-1)^{\binom{n}{2}} i^n
\left(\frac{1+i}{\sqrt{2}}\right)^{w(D)+\sum_{k=1}^n (-1)^k a_k}
\, q.
\]
\end{theorem}

\begin{proof}
We evaluate the Kauffman bracket using
Corollary~\ref{cor:Kbfinal} at the primitive eighth root of unity
\[
A=\frac{1+i}{\sqrt{2}}=e^{i\frac{\pi}{4}}.
\]
For this substitution we have
\[
A^2=i, \qquad A^4=-1, \qquad -A^2-A^{-2}=0.
\]

For an alternating rational diagram with $a_k>0$,
the twisting sign of the $k$th twist box is
\[
\sign_t(a_k)=(-1)^{k-1}.
\]
By Remark~\ref{rem:gamma},
\[
\gamma_k(a_1,\ldots,a_k)
=
i^{(-1)^{k-1}} a_k
=
(-1)^{k-1} i\, a_k.
\]

Hence Corollary~\ref{cor:Kbfinal} becomes
\[
\begin{bmatrix}
\kappa_n\\
\kappa_{n-1}
\end{bmatrix}
=
\begin{bmatrix}
(-1)^{n-1} i\, a_n & 1\\
1 & 0
\end{bmatrix}
\cdots
\begin{bmatrix}
i\, a_1 & 1\\
1 & 0
\end{bmatrix}
\begin{bmatrix}
1\\
0
\end{bmatrix}.
\]

Therefore $\kappa_n/\kappa_{n-1}$ equals the continued fraction
\[
[(-1)^{n-1} i\, a_n,
(-1)^{n-2} i\, a_{n-1},
\ldots,
-i\, a_2,
i\, a_1].
\]

By the Euler–Minding formulas
(see~\cite[formula~(1.3)]{He}),
the numerator $\kappa_n$ is given by
\[
\kappa_n
=
\sum_{S\subseteq_e \{1,\ldots,n\}}
\prod_{k\notin S} (-1)^{k-1} i\, a_k=\sum_{S\subseteq_e \{1,\ldots,n\}}\frac{(-1)^{\binom{n}{2}}i^na_1\cdots a_n}{\prod_{k\in S} ((-1)^{k-1}ia_k)},
\]
where $S\subseteq_e$ denotes the even subsets,
i.e.\ unions of disjoint intervals of even cardinality.
Since each subset $S$ can be partitioned into disjoint subsets of the form $\{j,j+1\}$, and the terms in $\prod_{k\in S} ((-1)^{k-1}ia_k$ corresponding to $j$ and $j+1$ can be simplified as
$$
(-1)^{j-1}ia_j\cdot (-1)^jia_{j+1}=a_ja_{j+1},
$$
we have
$$
\sum_{S\subseteq_e \{1,\ldots,n\}}\frac{(-1)^{\binom{n}{2}}i^na_1\cdots a_n}{\prod_{k\in S} ((-1)^{k-1}ia_k)}=
(-1)^{\binom{n}{2}}i^n\sum_{S\subseteq_e \{1,\ldots,n\}}\frac{a_1\cdots a_n}{\prod_{k\in S} a_k}.
$$
The final summation equals the numerator of the continued fraction
$[a_1,a_2,\ldots,a_n]$,
which is the denominator $q$ of
$[0,a_1,\ldots,a_n]$.
Thus
\[
\kappa_n
=
(-1)^{\binom{n}{2}} i^n q.
\]

It remains to evaluate $\omega_n$ at $A=\frac{1+i}{\sqrt{2}}$.
Using $A^2=i$ and $-A^2-A^{-2}=0$,
a direct calculation gives
\[
\omega_n(a_1,\ldots,a_n)
=
\left(\frac{1+i}{\sqrt{2}}\right)^{\sum_{k=1}^n (-1)^k a_k}.
\]

Therefore
\[
K_n
=
(-1)^{\binom{n}{2}} i^n
\left(\frac{1+i}{\sqrt{2}}\right)^{\sum_{k=1}^n (-1)^k a_k}
q.
\]

Finally, since $A^4=-1$, we have $-A^{-3}=A$, hence
\[
(-A^{-3})^{w(D)}=A^{w(D)}
=
\left(\frac{1+i}{\sqrt{2}}\right)^{w(D)}.
\]
Multiplying by this factor yields the stated formula for $V_L(-1)$.
\end{proof}

\begin{remark}{\em 
According to Murasugi~\cite[Proposition~11.2.7]{Mu},
every link $L$ satisfies
\[
V_L(-1)=(-1)^{\alpha(L)-1}\Delta_L(-1),
\]
where $\alpha(L)$ is the number of components and $\Delta_L$
is the Alexander polynomial.
By~\cite[Proposition~6.1.5]{Mu}, for a knot $K$,
$|\Delta_K(-1)|$ equals the determinant of $K$.
Cromwell~\cite[Theorem~8.7.7]{Cromwell} shows that a rational link
encoded by $p/q$ has determinant $|q|$
(with the convention that Cromwell's numerator and denominator are interchanged relative to ours).
Theorem~\ref{thm:t=-1} is therefore consistent with these classical results
and determines the precise phase factor relating $V_L(-1)$ and $q$.
}
\end{remark}

\begin{remark}{\em An easy consequence of Corollary\ref{cor:T41} and of
    the previous remark is that
    the number of spanning trees of the Tait graph is the absolute value
    of the determinant of an alternating rational link. Counting the
    spanning trees in precisely these graphs has been subject of recent
    interest in the work of S. H. Chan, A. Kontorovich, and
    I. Pak~\cite{Cha-Ko-Pak}, see also~\cite{Bier}.  
}
\end{remark}

Henceforth we assume that $d\neq 0$.
In this case, Corollary~\ref{cor:qA} may be restated as
\begin{equation}
\label{eq:qA}
A^{2\varepsilon}[m]_{-A^{4\varepsilon}}
=
\frac{(-1)^m A^{4m\varepsilon}-1}{d}.
\end{equation}

Using~\eqref{eq:qA}, Corollary~\ref{Cor_K_n} takes the following form.

\begin{corollary}
\label{cor:K_nnz}
Assume $d\neq 0$.
Then the Kauffman bracket
\[
K_n = K([0,a_1,a_2,\ldots,a_n])
\]
is the total weight of all tilings of a unit width path of length $n$ by
$1\times1$ and $2\times1$ tiles, with weights assigned as follows:
\begin{enumerate}
\item A $1\times1$ tile in position $1$ has weight
$((-A)^{3a_1}-A^{-a_1})/d + A^{-a_1}d$;
\item A $1\times1$ tile in position $k>1$ has weight
$
((-A)^{3(-1)^{k-1}a_k}-A^{(-1)^k a_k})/d$;
\item A $2\times1$ tile covering $(k,k+1)$ has weight
$
(-1)^{a_k}A^{(-1)^{k+1}(3a_k+a_{k+1})}$.
\end{enumerate}
\end{corollary}

Indeed, for $k\ge1$, equation~\eqref{eq:qA} implies
\[
A^{(-1)^{k-1}(2\varepsilon_k-a_k)}
[|a_k|]_{-A^{(-1)^{k-1}4\varepsilon_k}}
=
A^{(-1)^k a_k}\,
\frac{(-1)^{a_k}A^{4(-1)^{k-1}a_k}-1}{d},
\]
which is exactly the replacement used for the $1\times1$ tile weights.

Using~\eqref{eq:qA}, the parameters $\gamma_k(a_1,\ldots,a_k)$ may be simplified as
\begin{align}
\gamma_k(a_1,\ldots,a_k)
&=
A^{2(-1)^{k-1}\varepsilon_k}
[|a_k|]_{-A^{4(-1)^{k-1}\varepsilon_k}}
(-A^{4(-1)^{k-1}})^{a_1+\cdots+a_{k-1}}
\nonumber\\
&=
\frac{(-1)^{a_k}A^{4(-1)^{k-1}a_k}-1}{d}\,
(-A^{4(-1)^{k-1}})^{a_1+\cdots+a_{k-1}}.
\label{eq:gammaA}
\end{align}

Keeping in mind $t=A^{-4}$, equation~\eqref{eq:gammaA} may be rewritten as
\begin{equation}
\label{eq:gammat}
\gamma_k(a_1,\ldots,a_k)
=
\frac{(-t)^{(-1)^k(a_1+\cdots+a_k)}-
      (-t)^{(-1)^k(a_1+\cdots+a_{k-1})}}{d}.
\end{equation}
This yields a more compact expression for $\kappa_n$ in
Corollary~\ref{cor:Kbfinal}.

Using~\eqref{eq:qA}, the substitution rule~\eqref{eq:Msubs} may be rewritten as
\[
M_k
\longmapsto
\left[\begin{matrix}
\dfrac{1-(-A^4)^{(-1)^k a_k}}{d} & (-A^4)^{(-1)^k a_k}\\[4pt]
1 & 0
\end{matrix}\right].
\]
Substituting $t=A^{-4}$ (so that $-A^4=-t^{-1}$) and simplifying gives the
following consequence of Theorem~\ref{thm:lrk}.

\begin{corollary}
\label{cor:lrk}
Assume $d\neq 0$.
Then
\[
K_n
=
(-A)^{3\sum_{k=1}^n (-1)^{k-1}a_k}\,\rho_n,
\]
where
\[
\rho_n
=
\left[\begin{matrix}1&0\end{matrix}\right]
\left[\begin{matrix}
\dfrac{1-(-t)^{(-1)^{n-1}a_n}}{d} & (-t)^{(-1)^{n-1}a_n}\\[4pt]
1&0
\end{matrix}\right]
\cdots
\left[\begin{matrix}
\dfrac{1-(-t)^{a_1}}{d} & (-t)^{a_1}\\[4pt]
1&0
\end{matrix}\right]
\left[\begin{matrix}1\\ d\end{matrix}\right].
\]
\end{corollary}

Exploring substitutions of roots of unity into the Jones polynomial is
a subject of intensive ongoing research, see the third remark on page
383 in~\cite{O}. In this exploration the following result may be helpful.
\begin{theorem}
\label{thm:rootunity}
Let $\delta>1$ be an integer such that $\delta\mid (|a_1|-1)$ and
$\delta\mid a_i$ for $i=2,\ldots,n$.
Let $\rho$ be a $\delta$th root of unity.
Then the Jones polynomial of the rational link encoded by
$[0,a_1,a_2,\ldots,a_n]$, evaluated at $t=-\rho$, is a $(4\delta)$th root of unity.
\end{theorem}

\begin{proof}
If $\delta\mid m$, then for any $\delta$th root of unity $\rho$ we have
\[
[m]_q=\frac{q^m-1}{q-1}=0
\quad\text{at } q=\rho,
\]
and likewise at $q=\rho^{-1}$.

Now substitute $(t=)A^{-4}=-\rho$.
Since $\delta\mid a_k$ for $k\ge2$, we have
\[
[|a_k|]_{-A^{\pm 4}}=0
\quad\text{for } k\ge2.
\]
On the other hand, $\delta\mid (|a_1|-1)$ implies
\[
[|a_1|]_{-A^{\pm4}}=1.
\]

Under this substitution, in the matrix product of
Corollary~\ref{cor:Kbfinal} we therefore obtain
\[
\gamma_k=0 \quad\text{for } k\ge2,
\qquad
\gamma_1=A^2.
\]
Hence
\[
\begin{bmatrix}
\kappa_n\\
\kappa_{n-1}
\end{bmatrix}
=
\begin{bmatrix}
0&1\\
1&0
\end{bmatrix}^{n-1}
\begin{bmatrix}
A^2&1\\
1&0
\end{bmatrix}
\begin{bmatrix}
1\\
d
\end{bmatrix}=
\begin{bmatrix}
0&1\\
1&0
\end{bmatrix}^{n-1}
\begin{bmatrix}
-A^{-2}\\
1
\end{bmatrix}.
\]

A straightforward computation now gives
$
\kappa_n=(-1)^n A^{(-1)^n-1}$.

The remaining factors relating $\kappa_n$ to $V_L(t)$,
namely $(-A^{-3})^{w(D)}$ and $\omega_n$,
are products of powers of $A$.
Therefore $V_L(-\rho)$ is a power of $\pm A$.
Since $t=A^{-4}=-\rho$ and $\rho^\delta=1$,
it follows that $A^{4\delta}=1$.
Hence $V_L(-\rho)$ is a $(4\delta)$th root of unity.
\end{proof}

\section{A combinatorial formula for the Jones polynomial of a rational link}
\label{sec:LR}

In this section we derive a variant of the Jones polynomial formula of
Lawrence and Rosenstein~\cite[Theorem~4.2]{La-Ro}
and show that it admits a natural combinatorial reformulation.

Following the notation of Lawrence and Rosenstein, we introduce
\begin{equation}
u=(-t^{-1})^{1/2}= iA^2,
\end{equation}
and define
\begin{equation}
m_k = (-1)^{k-1} a_k.
\end{equation}
By~\eqref{eq:tsign}, the sign of $m_k$,
\[
\sign(m_k)=(-1)^{k-1}\sign(a_k),
\]
coincides with the twisting sign of the crossings in the twist box
associated with $a_k$.

\begin{definition}
Let $D$ be a rational link diagram encoded by
\[
[0,a_1,a_2,\ldots,a_n]
=
[0,m_1,-m_2,\ldots,(-1)^{n-1}m_n].
\]
The \emph{Lawrence–Rosenstein polynomial} of $D$ is defined by
\[
R_D(t)
=
u^{\sum_{k=1}^n m_k}
\cdot
\frac{d^n\, V_{L(D)}(t)}
{(-A)^{- 3w(D)+3\sum_{k=1}^n m_k }}.
\]
\end{definition}

The polynomial $R_D(t)$ is a Laurent polynomial in $\sqrt{t}$,
just like the Jones polynomial, as shown by the following result.

\begin{proposition}
\label{prop:lr}
Let $D_n$ be a rational link diagram encoded by
\[
[0,a_1,\ldots,a_n]
=
[0,m_1,-m_2,\ldots,(-1)^{n-1}m_n].
\]
Then
\[
R_D(t)
=
\begin{bmatrix}1 & 0\end{bmatrix}
\left[
\begin{matrix}
u^{m_n}-u^{-m_n} & u^{-m_n} d \\
u^{m_n} d & 0
\end{matrix}
\right]
\cdots
\left[
\begin{matrix}
u^{m_1}-u^{-m_1} & u^{-m_1} d \\
u^{m_1} d & 0
\end{matrix}
\right]
\begin{bmatrix}1\\ d\end{bmatrix}.
\]
\end{proposition}

\begin{proof}
Define
\[
\widetilde{R}_D(t)
=
u^{-\sum_{k=1}^n m_k} R_D(t).
\]
We first show that
$$
\widetilde{R}_D(t)
=
\begin{bmatrix}1 & 0\end{bmatrix}
\left[
\begin{matrix}
1-(-t)^{m_n} & (-t)^{m_n} d \\
d & 0
\end{matrix}
\right]
\cdots
\left[
\begin{matrix}
1-(-t)^{m_1} & (-t)^{m_1} d \\
d & 0
\end{matrix}
\right]
\begin{bmatrix}1\\ d\end{bmatrix}.
$$
When $d\neq0$, this follows directly from
Corollary~\ref{cor:lrk}, since
\(
\widetilde{R}_D(t)=d^n \rho_n.
\)
The general case then follows due to the fact that both sides of the above are Laurent polynomials of $A$. 

Now multiply both sides by $u^{\sum_{k=1}^n m_k}$.
Since
\[
(-t)^{m_k}=u^{-2m_k},
\]
each matrix factor
\[
\begin{bmatrix}
1-u^{-2m_k} & u^{-2m_k} d\\
d & 0
\end{bmatrix}
\]
becomes
\[
u^{m_k}
\begin{bmatrix}
1-u^{-2m_k} & u^{-2m_k} d\\
d & 0
\end{bmatrix}
=
\begin{bmatrix}
u^{m_k}-u^{-m_k} & u^{-m_k} d\\
u^{m_k} d & 0
\end{bmatrix}.
\]
Substituting these into the product yields the stated formula.
\end{proof}

Introducing the notation
\[
L=
\begin{bmatrix}
1 & 0\\
d & 0
\end{bmatrix}
\]
and defining the $\star$–operation by
\[
\begin{bmatrix}
a_{11} & a_{12}\\
a_{21} & a_{22}
\end{bmatrix}^{\star}
=
\begin{bmatrix}
a_{11} & -a_{21}\\
-a_{12} & a_{22}
\end{bmatrix},
\]
the formula of Proposition~\ref{prop:lr} may be rewritten as
\begin{equation}
\label{eq:lr}
R_D(t)
=
\begin{bmatrix}1 & 0\end{bmatrix}
(u^{m_n}L-u^{-m_n}L^\star)
\cdots
(u^{m_1}L-u^{-m_1}L^\star)
L
\begin{bmatrix}1\\0\end{bmatrix}.
\end{equation}

\begin{lemma}
\label{lem:lprod}
Let
\[
E=
\begin{bmatrix}
1 & 0\\
0 & 0
\end{bmatrix}.
\]
Then the matrices $E$, $L$ and $L^\star$ satisfy
\[
E^2=E, \quad L^2=L, \quad (L^\star)^2=L^\star,\quad L^\star L=(1-d^2)\,E
\]
and 
\[
LE=L, \quad EL=E, \quad
L^\star E=E, \quad EL^\star=L^\star.
\]
\end{lemma}

\noindent
The verification is straightforward.

The polynomial $R_D(t)$ admits a combinatorial expansion
that is most naturally expressed in terms of interval
decompositions of subsets of $\{1,\ldots,n\}$.

\begin{definition}
Let $S$ be a set of positive integers.
An \emph{interval of $S$} is a maximal subset of $S$
of the form
\[
\{k,k+1,\ldots,\ell\}
\]
for some integers $k\le \ell$.
Equivalently, it is a maximal block of consecutive integers contained in $S$.
We denote the set of intervals of $S$ by $I(S)$.
\end{definition}

For example, if
\[
S=\{2,3,5,6,7,9\},
\]
then
\[
I(S)=\{[2,3],\, [5,7],\, [9,9]\}.
\]

\begin{theorem}
\label{thm:lrc}
We have
\[
R_D(t)
=
\sum_{S\subseteq \{1,\ldots,n\}}
\left(
\prod_{k\notin S} u^{m_k}
\right)
\left(
\prod_{k\in S} (-u^{-m_k})
\right)
(u^2+u^{-2}-1)^{|I(S)|}.
\]
\end{theorem}

\begin{proof}
Expanding the matrix product in~\eqref{eq:lr}
produces $2^n$ terms.
For each subset $S\subseteq\{1,\ldots,n\}$,
select the factor $-u^{-m_k}L^\star$ whenever $k\in S$,
and $u^{m_k}L$ whenever $k\notin S$.

For example, if $n=6$ and $S=\{2,4,5,6\}$ (so $|I(S)|=2$),
the corresponding contribution is
\[
\begin{bmatrix}1 & 0\end{bmatrix}
u^{m_1}L(-u^{-m_2}L^\star)
u^{m_3}L
(-u^{-m_4}L^\star)
(-u^{-m_5}L^\star)
(-u^{-m_6}L^\star)
L
\begin{bmatrix}1\\0\end{bmatrix}.
\]
Factoring out the scalar terms gives
\[
u^{m_1-m_2+m_3-m_4-m_5-m_6}
(-1)^{|S|}
\begin{bmatrix}1 & 0\end{bmatrix}
L L^\star L (L^\star)^3 L
\begin{bmatrix}1\\0\end{bmatrix}.
\]
By Lemma~\ref{lem:lprod}, 
$$
\begin{bmatrix}1 & 0\end{bmatrix}L L^\star L (L^\star)^3 L\begin{bmatrix}1\\0\end{bmatrix}=
(1-d^2)^2 \begin{bmatrix}1 & 0\end{bmatrix}L\begin{bmatrix}1\\0\end{bmatrix}=(1-d^2)^2.
$$

In general, each maximal block of consecutive indices in $S$
produces one factor $L^\star L$.
Hence for each $S$ the corresponding matrix product reduces to
\[
(1-d^2)^{|I(S)|}.
\]
Therefore the contribution of $S$ equals
\[
\left(
\prod_{k\notin S} u^{m_k}
\right)
\left(
\prod_{k\in S} (-u^{-m_k})
\right)
(1-d^2)^{|I(S)|}.
\]

Finally, since
\[
1-d^2
=
1-(A^{-2}+A^2)^2
=
1-(iu^{-1}-iu)^2
=
u^2+u^{-2}-1,
\]
the stated formula follows.
\end{proof}

It is worth noting that
\[
u^2+u^{-2}-1=-\frac{1+t+t^2}{t},
\]
so Theorem~\ref{thm:lrc} may be restated for
\(
\widetilde{R}_D(t)=u^{-\sum_{k=1}^n m_k}R_D(t)
\)
as
\[
\widetilde{R}_D(t)
=
\sum_{S\subseteq \{1,\ldots,n\}}
(-1)^{|S|+|I(S)|}
\prod_{k\in S} (-t)^{m_k}
\left(\frac{1+t+t^2}{t}\right)^{|I(S)|}.
\]

\begin{example}{\em 
For $A=\exp(\pi i/3)$, $t=A^{-4}=\exp(-4\pi i/3)=\exp(2\pi i/3)$ and $1+t+t^2=0$.
Hence the only nonzero contribution in
Theorem~\ref{thm:lrc} comes from $S=\emptyset$.
Therefore
\[
\widetilde{R}_D(\exp(2\pi i/3))=1.
\]
Since $(-A)^3=-\exp(\pi i)=1$ and
$d=-A^2-A^{-2}=1$,
it follows that
\[
V_D(\exp(2\pi i/3))=1.
\]
}
\end{example}

\section{The Tutte polynomial of the Tait graph of an alternating rational link diagram}
\label{sec:Taitalt}

All results of this paper are directly applicable in the special case
when the rational link is represented by an alternating diagram. For such diagrams there is,
however, another way to compute the Kauffman bracket: one may compute
the Tutte polynomial of the Tait graph and then apply
Corollary~\ref{cor:alternating}.  In this section we outline this
approach and obtain a generalization of~\cite[Theorem~3.4]{Q}.

It should be noted that to apply Corollary~\ref{cor:alternating} we
only need the evaluation of $T(G;X,Y)$ at $Y=X^{-1}$. Under this
specialization, many of the formulas below simplify further.

Without loss of generality we may assume that the rational
link diagram in standard form is represented by the continued fraction
$[0,a_1,a_2,\ldots,a_n]$, where
$a_1,a_2,\ldots,a_n$ are all positive. In this case all edges of
the Tait graph have positive sign. The ordinary Tutte polynomial may
then be obtained from the signed Tutte polynomial by setting
\[
x_+=1, \qquad y_+=1, \qquad X_+=X, \qquad Y_+=Y.
\]
Consequently, equations~\eqref{eq:ycondsu} and~\eqref{eq:xcondsv}
become
\begin{equation}
\label{eq:condsuvalt}
u=X-1
\qquad\text{and}\qquad
v=Y-1.
\end{equation}

We may still apply Theorems~\ref{Ttensor} and~\ref{structure} to obtain
the Tutte polynomial of the Tait graph from the colored Tutte
polynomial $T_n$ of the core graph $G_n$ by using suitable
specializations of the substitution rules
\eqref{eq:Pm-e}--\eqref{eq:Pm*TL}.

For an even index $2j$ we have $\widehat{Q}_{2j}=\widehat{P}_{a_{2j}}$
and the following substitution rules:
\begin{eqnarray}
\label{eq:Xe+}
X_{2j} &\mapsto& T(\widehat{Q}_{2j}-e)\mapsto X^{a_{2j}},\\
\label{eq:Ye+}
Y_{2j} &\mapsto& T(\widehat{Q}_{2j}/e)\mapsto [a_{2j}]_X + Y - 1,\\
\label{eq:xe+}
x_{2j} &\mapsto& T_L(\widehat{Q}_{2j},e)\mapsto 1,\\
\label{eq:ye+}
y_{2j} &\mapsto& T_C(\widehat{Q}_{2j},e)\mapsto [a_{2j}]_X .
\end{eqnarray}

Similarly, for an odd index $2j+1$ we have
$\widehat{Q}_{2j+1}=\widehat{P}^*_{a_{2j+1}}$ and
\begin{eqnarray}
\label{eq:Xo+}
X_{2j+1} &\mapsto& T(\widehat{Q}_{2j+1}-e)
      \mapsto [a_{2j+1}]_Y + X - 1,\\
\label{eq:Yo+}
Y_{2j+1} &\mapsto& T(\widehat{Q}_{2j+1}/e)\mapsto Y^{a_{2j+1}},\\
\label{eq:xo+}
x_{2j+1} &\mapsto& T_L(\widehat{Q}_{2j+1},e)\mapsto [a_{2j+1}]_Y,\\
\label{eq:yo+}
y_{2j+1} &\mapsto& T_C(\widehat{Q}_{2j+1},e)\mapsto 1 .
\end{eqnarray}

These substitutions determine the parameters $u_k$ and $v_k$:
\begin{equation}
\label{u_nalt}
u_k \mapsto
\begin{cases}
X^{a_{k-1}}, & \text{if $k$ is odd},\\
Y^{a_{k-1}}, & \text{if $k$ is even},
\end{cases}
\end{equation}

\begin{equation}
\label{v_nalt}
v_k \mapsto
\begin{cases}
[a_k]_Y, & \text{if $k$ is odd},\\
[a_k]_X, & \text{if $k$ is even}.
\end{cases}
\end{equation}

In particular, we set
\[
u_1=y_1\mapsto 1=X^0,
\qquad
v_1=x_1\mapsto [a_1]_Y.
\]

Substituting these rules into~\eqref{eq:Tmatrixu} yields the following
formula.

\begin{theorem}
\label{thm:Taitalt}
Let $G$ be the Tait graph of a standard rational link diagram encoded by
the continued fraction
\[
[0,a_1,a_2,\ldots,a_n],
\]
where $a_1,a_2,\ldots,a_n$ are positive.
Then the Tutte polynomial $T(G;X,Y)$ is given by
\[
T(G;X,Y)=
\begin{bmatrix}1 & 0\end{bmatrix}
A_n A_{n-1}\cdots A_1
\begin{bmatrix}
1\\
X-1
\end{bmatrix}.
\]

Here $a_0=0$, and
\[
A_k=
\begin{cases}
\begin{bmatrix}
[a_k]_Y & X^{a_{k-1}}\\
1 & 0
\end{bmatrix},
& \text{if $k$ is odd}, \\[16pt]

\begin{bmatrix}
[a_k]_X & Y^{a_{k-1}}\\
1 & 0
\end{bmatrix},
& \text{if $k$ is even}.
\end{cases}
\]
\end{theorem}

\begin{remark}
\emph{Substituting $Y=X^{-1}$ in Theorem~\ref{thm:Taitalt} yields
\[
A_k \mapsto
\begin{bmatrix}
[a_k]_{X^{(-1)^k}} &
X^{(-1)^{k-1}a_{k-1}}\\
1 & 0
\end{bmatrix}
\qquad\text{for all $k$.}
\]}
\end{remark}

It should be noted that even for odd $n$, the formula stated in
Theorem~\ref{thm:Taitalt} looks very different from the one stated
in~\cite[Theorem~3.4]{Q}, which may be rewritten in our notation as
follows.

\begin{theorem}[Qazaqzeh--Yasein--Abu-Qamar]
Let $G$ be the Tait graph of a standard rational link diagram encoded by
the continued fraction
\[
[0,a_1,a_2,\ldots,a_n],
\]
where $a_1,a_2,\ldots,a_n$ are positive and $n=2k+1$ is odd. Then the Tutte polynomial $T(G;X,Y)$ is
given by
\[
T(G;X,Y)=
\sum_{\substack{0\le l\le k\\
1\le i_1<i_2<\cdots<i_l\le k}}
\prod_{j=1}^l [a_{2i_j}]_X
\prod_{j=0}^l \big([c_j]_Y+X-1\big),
\]
where we set $i_0=0$, $i_{l+1}=k+1$, and
\[
c_j=a_{2i_j+1}+a_{2i_j+3}+\cdots+a_{2i_{j+1}-1},
\qquad j=0,1,\ldots,l .
\]
\end{theorem}

Our formula expresses $T(G;X,Y)$ directly in terms of the partial
denominators $a_j$, without introducing the auxiliary quantities $c_j$,
whose definition depends on the choice of a subset
$\{i_1,\ldots,i_l\}\subseteq\{1,\ldots,k\}$.
Since the Tutte polynomial of a graph is uniquely determined, the two
expressions must agree when $n$ is odd, as demonstrated by the following simple example.

\begin{example}
{\em
Consider the case $n=3$ with $a_3=1$.
By Theorem~\ref{thm:Taitalt} the Tutte polynomial of the Tait graph is
\begin{align*}
T(G;X,Y)
&=
\begin{bmatrix}1 & 0\end{bmatrix}
\begin{bmatrix}
1 & X^{a_2}\\
1 & 0
\end{bmatrix}
\begin{bmatrix}
[a_2]_X & Y^{a_1}\\
1 & 0
\end{bmatrix}
\begin{bmatrix}
[a_1]_Y & 1\\
1 & 0
\end{bmatrix}
\begin{bmatrix}
1\\
X-1
\end{bmatrix}\\[6pt]
&=
\begin{bmatrix}
1 & X^{a_2}
\end{bmatrix}
\begin{bmatrix}
[a_2]_X & Y^{a_1}\\
1 & 0
\end{bmatrix}
\begin{bmatrix}
[a_1]_Y+X-1\\
1
\end{bmatrix}\\[6pt]
&=
\begin{bmatrix}
[a_2]_X+X^{a_2} & Y^{a_1}
\end{bmatrix}
\begin{bmatrix}
[a_1]_Y+X-1\\
1
\end{bmatrix}\\[6pt]
&=
\big([a_2]_X+X^{a_2}\big)\big([a_1]_Y+X-1\big)+Y^{a_1}\\
&=
[a_2+1]_X\big([a_1]_Y+X-1\big)+Y^{a_1}.
\end{align*}

Translating the notation of~\cite{Q} into ours,
\cite[Corollary~3.5]{Q} gives
\begin{align*}
T(G;X,Y)
&=
X[a_2]_X\big([a_1]_Y+X-1\big)+[a_1+1]_Y+X-1\\
&=
\big([a_2+1]_X-1\big)\big([a_1]_Y+X-1\big)+[a_1+1]_Y+X-1\\
&=
[a_2+1]_X\big([a_1]_Y+X-1\big)
-\big([a_1]_Y+X-1\big)+[a_1+1]_Y+X-1\\
&=
[a_2+1]_X\big([a_1]_Y+X-1\big)+Y^{a_1},
\end{align*}
which coincides with the formula obtained above.
}
\end{example}

Using the substitution rules \eqref{u_nalt} and \eqref{v_nalt},
Theorem~\ref{thm:coremain} has the following consequence.

\begin{corollary}
\label{cor:coremainalt}
Let $G$ be the Tait graph of a standard alternating rational link
diagram encoded by the continued fraction
$[0,a_1,a_2,\ldots,a_n]$, where 
$a_1,a_2,\ldots,a_n$ are positive.
Define the Laurent polynomial $P(G;X,Y)$ by
\[
P(G;X,Y)=
\begin{cases}
X^{-(a_{n-1}+a_{n-3}+\cdots)}\,T(G;X,Y), & \text{if $n$ is odd},\\
Y^{-(a_{n-1}+a_{n-3}+\cdots)}\,T(G;X,Y), & \text{if $n$ is even}.
\end{cases}
\]
Then
\[
P(G;X,Y)=
\begin{bmatrix}
1 & 0
\end{bmatrix}
B_n B_{n-1}\cdots B_1
\begin{bmatrix}
1\\
X-1
\end{bmatrix}.
\]

Here the matrix $B_k$ is given by
\[
B_k=
\begin{cases}
\begin{bmatrix}
\dfrac{[a_k]_Y\,Y^{a_{k-2}+a_{k-4}+\cdots}}
     {X^{a_{k-1}+a_{k-3}+\cdots}} & 1\\
1 & 0
\end{bmatrix}, & \text{if $k$ is odd},\\[18pt]
\begin{bmatrix}
\dfrac{[a_k]_X\,X^{a_{k-2}+a_{k-4}+\cdots}}
     {Y^{a_{k-1}+a_{k-3}+\cdots}} & 1\\
1 & 0
\end{bmatrix}, & \text{if $k$ is even}.
\end{cases}
\]
\end{corollary}

\begin{remark}
{\em Substituting $Y=X^{-1}$ in Corollary~\ref{cor:coremainalt} yields
  $$
P(G;X,X^{-1})=X^{(-1)^n(a_{n-1}+a_{n-3}+\cdots )}T(G;X,X^{-1})\quad\mbox{and}
$$
\begin{align*}
B_k&\mapsto
\left[\begin{matrix}
[a_k]_{X^{(-1)^k}}\cdot X^{(-1)^k (a_1+a_2+\cdots+a_{k-1})}&1\\
1&0\\
  \end{matrix}\right]\\
&=
\left[\begin{matrix}
[a_1+\cdots+a_k]_{X^{(-1)^k}}-[a_1+\cdots+a_{k-1}]_{X^{(-1)^k}}&1\\
1&0\\
  \end{matrix}\right]
\end{align*}
}
\end{remark}  

Using the substitution rules \eqref{u_nalt} and \eqref{v_nalt},
Theorem~\ref{thm:lr} has the following consequence.   

\begin{corollary}
\label{cor:lralt}
Let $G$ be the Tait graph of a standard alternating rational link
diagram encoded by the continued fraction
$[0,a_1,a_2,\ldots,a_n]$, where 
$a_1,a_2,\ldots,a_n$ are positive.
Then the Laurent polynomial
\[
R(G;X,Y)
=
\frac{T(G;X,Y)}
{\prod_{1\le j\le \lfloor n/2\rfloor} X^{a_{2j}}
 \prod_{1\le j\le \lfloor (n+1)/2\rfloor} Y^{a_{2j-1}}}
\]
is given by
$$
R(G;X,Y)=\begin{bmatrix}
1 & 0\\
  \end{bmatrix}
C_n C_{n-1}\cdots C_1
\begin{bmatrix}
1\\
X-1\\
\end{bmatrix}.
$$
Here
$$
C_k=
\begin{cases}
\begin{bmatrix}
\frac{[a_{k}]_X}{X^{a_k}}& \frac{1}{X^{a_k}} \\
1& 0\\
\end{bmatrix},&\text{if $k$ is even;}\\[18pt]
\begin{bmatrix}
\frac{[a_k]_Y}{Y^{a_k}}& \frac{1}{Y^{a_k}} \\
1& 0\\
\end{bmatrix},&\text{if $k$ is odd.}\\
\end{cases}  
$$
\end{corollary}  

\begin{remark}
{\em Substituting $Y=X^{-1}$ in Corollary~\ref{cor:lralt} yields
$$R(G;X,X^{-1})=T(G;X,X^{-1})\cdot \prod_{1\leq k\leq n} X^{(-1)^{k-1}
    a_{k}}\quad\mbox{and}$$ 
  $$
C_k\mapsto \left[\begin{matrix}
[a_{k}]_{X^{(-1)^k}} X^{(-1)^{k-1}a_k}& X^{(-1)^{k-1} a_k} \\
1& 0\\
\end{matrix}\right].
  $$
}
\end{remark}

\section{Computing the writhe of a rational link diagram}
\label{sec:writhe}

Inspired by a finite automaton introduced in~\cite{DEH} and by the
illustrations in~\cite{Du}, we introduce two finite automata that can
be used to compute the writhe of any oriented rational link diagram in
standard form, together with a third automaton that computes the number
of connected components. These automata allow us to simplify the proofs
and generalize the the results of Qazaqzeh, Yasein, and Abu-Qamar on the
writhe of rational link diagrams~\cite[Propositions~4.2 and 4.3]{Q}.
Our approach is applicable to rational link diagrams with any number of
twist boxes, whereas in~\cite{Q} the authors require to have an odd
number of twist boxes. This is the point where our approach is genuinely more
general. Extending the 
method beyond alternating diagrams is in fact not essential, 
because of Lemma~\ref{lem:writhe} below. Given an oriented link diagram
$\mathcal{D}$, we call the unoriented diagram obtained by simply
``forgetting'' the orientation of its strands the {\em underlying
unoriented link diagram}.   

\begin{lemma}
\label{lem:writhe}  
Let $\mathcal{D}$ be an oriented rational link diagram whose
underlying unoriented diagram is in standard form and encoded by the
continued fraction $[0,a_1,a_2,\ldots,a_n]$. Let $\mathcal{D}'$ be the
alternating oriented link diagram obtained from $\mathcal{D}$ by
changing all overcrossings to undercrossings in the twist boxes
associated with negative partial denominators $a_i$. Equivalently,
$\mathcal{D}'$ is the alternating oriented link diagram whose  underlying
unoriented diagram of $\mathcal{D}'$ is encoded by
$[0,|a_1|,|a_2|,\ldots,|a_n|]$.

If the writhe of $\mathcal{D}'$ is
$w(\mathcal{D'})=\sum_{i=1}^n s_i\cdot |a_i|$ such that $s_i\in\{1,-1\}$
holds for all $i$, then 
$$w(\mathcal{D})=\sum_{i=1}^n s_i\cdot a_i.$$
\end{lemma}

\begin{proof}
The crossings in the twist box
represented by $a_i$ contribute $s_i a_i$ to the writhe for some
$s_i\in\{1,-1\}$. The same also holds for $\mathcal{D}'$.
Thus it suffices to show that the same coefficients $s_i$ occur in both
expressions.

Consider the operation of changing all overcrossings to undercrossings
in a twist box corresponding to some $a_i<0$. This move changes the
link represented by the diagram, but it does not change the number of
components or the orientations of the strands. Consequently the sign
of every crossing in that twist box changes to its opposite. At the
same time we replace $a_i<0$ with $|a_i|>0$. These two changes cancel
each other in the writhe computation. If the original twist box
contributes $s_i a_i$ to the writhe of $\mathcal{D}$, then the modified
twist box contributes $s_i |a_i|$ to the writhe of $\mathcal{D}'$.
\end{proof}

Consider an unoriented rational link diagram in standard form encoded by the 
continued fraction $[0,a_1,a_2,\ldots,a_n]$.
We number its strands from $1$ to $4$ from top to bottom, as shown in
Figure~\ref{fig:samplelink}.  
\begin{figure}[htb]
\begin{center}
  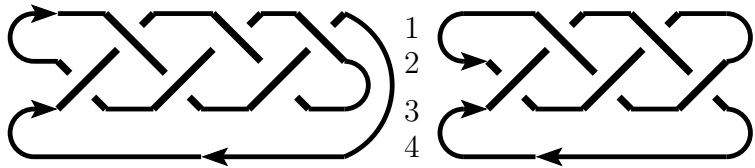
\end{center}
\caption{Two sample links}
\label{fig:samplelink}
\end{figure}
Next we turn our unoriented link diagram into an oriented link diagram.
Without loss of generality we may assume that the lowest strand is
oriented from right to left. At the left end, record each oriented
connection from strand $i$ to strand $j$ by $(i,j)$.  With the bottom
connection fixed as $(4,3)$, the two initial states are $(2,1)(4,3)^*$ and $(1,2)(4,3)^*$, where a
star records that an even number of twist boxes has been read.
These two possible initial states are underlined in
Figure~\ref{fig:big_automaton}. We move to a new state each time we read
a twistbox, the lack of the star in the code of a state indicates that
an odd number of twistboxes have been read. The rest of the code of the
current state is
$(i,j)(k,l)$ if parsing the link so far we have moved from strand $i$ to
strand $j$ and from strand $k$ to strand $l$.
For example, parsing the sample link on the right-hand
side of Figure~\ref{fig:samplelink} yields the oriented path
$$
(1,2)(4,3)^* \rightarrow (1,3)(4,2) \rightarrow (2,3)(4,1)^* \rightarrow
(3,2)(4,1)
\rightarrow (3,1)(4,2)^* \rightarrow (2,1)(4,3),
$$
We stop when we have read all twistboxes. The appropriate right-hand
closure, shown in Figures~\ref{fig:rationallink} and
~\ref{fig:samplelink}, depends only on the parity of the stopping state.  
\begin{figure}[htb]
\begin{center}
  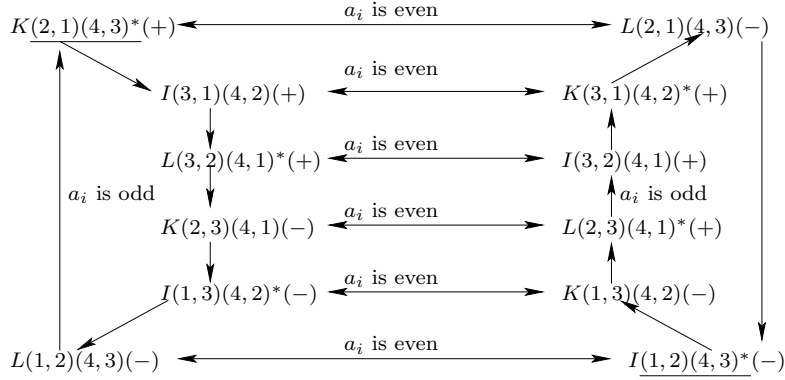
\end{center}
\caption{An automaton that helps compute the writhe and the number of
connected components of a rational link diagram}
\label{fig:big_automaton}
\end{figure}
Reading an even $a_i$ toggles the parity of a state without changing
its ordered pairs; reading an odd $a_i$ follows the corresponding
directed edge in Figure~\ref{fig:big_automaton}.  The labels $K$, $L$,
and $I$ indicate, respectively, a knot, a two-component link, or an
orientation-incompatible closure.  The main results state that these
labels and the displayed crossing signs are correct.

\begin{theorem}
\label{thm:components}
Consider an unoriented rational link diagram encoded by
$[0,a_1,a_2,\ldots,a_n]$. The left ends of this diagram may be oriented as
indicated in a starting state of the automaton shown in
Figure~\ref{fig:big_automaton} if and only if reading the vector
$(a_1,a_2,\ldots,a_n)$ results in a final state that is not marked
$I$. The unoriented link diagram represents a knot if and only if the
corresponding final state in Figure~\ref{fig:big_automaton} is marked
$K$ or $I$.
\end{theorem}

\begin{proof}
We prove this statement by presenting a ``quotient'' of the
automaton shown in Fig.~\ref{fig:big_automaton}. This quotient, shown in
Fig.~\ref{fig:c-automaton}, is obtained by identifying the state
$(a,b)(4,c)^*$ with the state $(b,a)(4,c)^*$ and the state $(a,b)(4,c)$
with the state $(b,a)(4,c)$.

\begin{figure}[htb]
\begin{center}
  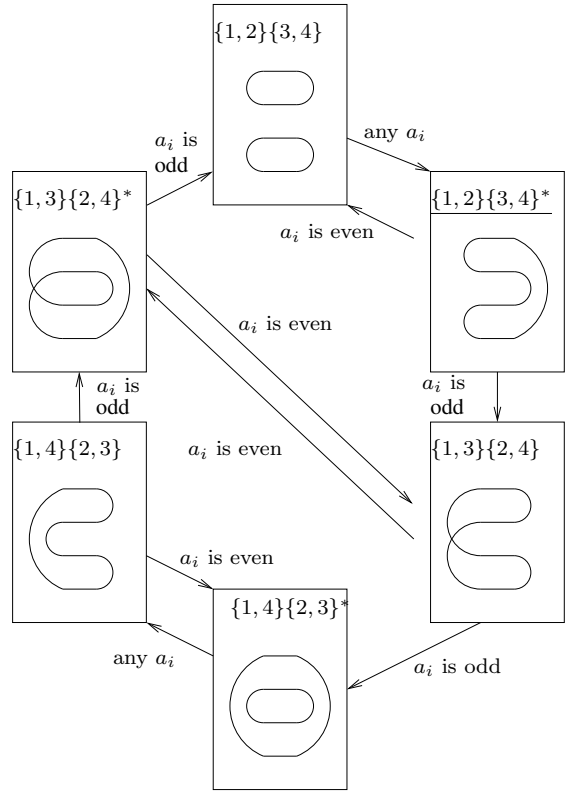
\end{center}
\caption{A simplified automaton that helps compute the number of
connected components of a rational link diagram}
\label{fig:c-automaton}
\end{figure}

This identification corresponds to forgetting the orientation of the
strands, hence we label the states with unordered pairs of numbers.
Regardless of which state we are in that corresponds to the same pairing
of strands on the left, the number of connected components depends only
on the pairing of the strands on the right, which in turn depends only
on the parity of the number of twist boxes parsed so far. It depends
only on the parity of $a_i$ which pairing of the four strands we obtain
next. The visual left and right pairings of the strands at each state
are shown in Fig.~\ref{fig:c-automaton}. It is immediate to see which
states represent (unoriented) knots and which represent two-component
links. Each state of Fig.~\ref{fig:c-automaton} corresponds to a pair of
states in Fig.~\ref{fig:big_automaton}: for two-component links both
states represent valid orientations, whereas for knots one state
corresponds to a valid oriented knot diagram and the other to an
invalid final state.
\end{proof}

\begin{theorem}
\label{thm:writhe}
To compute the writhe of the corresponding oriented link diagram,
multiply each $a_i$ by the sign associated with the state reached
immediately before reading $a_i$.
\end{theorem}

\begin{proof}
By Lemma~\ref{lem:writhe}, it suffices to prove the statement for the
special case when all $a_i$ are positive, so the diagram is alternating.
To verify the signs, consider the
simplified automaton shown in Figure~\ref{fig:s-automaton}.

\begin{figure}[htb]
\begin{center}
  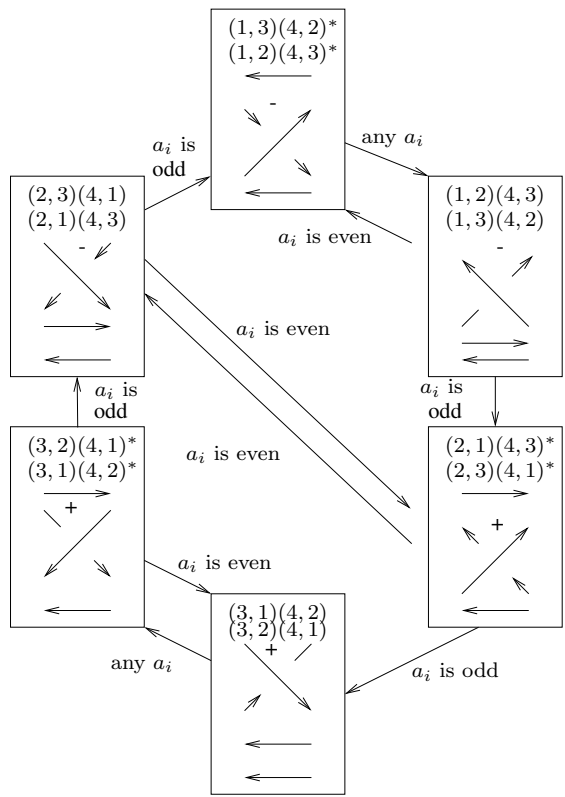
\end{center}
\caption{A simplified automaton that helps compute the writhe of a rational link diagram}
\label{fig:s-automaton}
\end{figure}

This automaton is obtained by merging pairs of states of the same
parity that represent the same orientation of the strands immediately
before the next twist box. For example, the states $(1,3)(4,2)^*$ and
$(1,2)(4,3)^*$ are the two even states in which strands $1$ and $4$
are oriented from right to left and strands $2$ and $3$ are oriented
from left to right. Because these are even states, the crossings in
the next twist box involve strands $2$ and $3$. When all $a_i$ are
positive, the crossings appear exactly as shown in the diagram
associated with the state. The description of the states
$(1,3)(4,2)$ and $(1,2)(4,3)$ is completely analogous, the only
difference being that the crossings in the next twist box involve
strands $1$ and $2$.

In general, state $(u_1,v_1)(u_2,v_2)$ is paired with state
$(u_1,v_2)(u_2,v_1)$, and state $(u_1,v_1)(u_2,v_2)^*$ is paired
with state $(u_1,v_2)(u_2,v_1)^*$. If we start from either state
of such a pair, after reading the next $a_i$ we arrive at one of two
states that are again paired. Thus the automaton shown in
Figure~\ref{fig:s-automaton} is also a ``homomorphic image'' of the
automaton shown in Figure~\ref{fig:big_automaton}.
\end{proof}

\vanish{We conclude this section with a relabeled variant of the automaton shown
in Figure~\ref{fig:big_automaton}. The pairs of permutations labeling
the states in Figure~\ref{fig:big_automaton_perm} extend the notation
introduced in~\cite{Q}. Our labeling of the strands is the same as
theirs. In~\cite{Q} each canonical oriented link diagram $D$ is represented by
a permutation
\[
\sigma_D=(23)^{a_1}(12)^{a_2}\cdots(12)^{a_{n-1}}(23)^{a_n}
\]
of the set $\{1,2,3\}$, where $n$ is an odd integer and
$[0,a_1,\ldots,a_n]$ is the continued fraction representing the
underlying unoriented rational link. In~\cite{Q} the partial
denominators $a_i$ are all positive integers, and multiplication of the
involutions is performed from left to right.

We extend this permutation labeling to all states of our automaton
shown in Figure~\ref{fig:big_automaton}. Allowing negative integers
$a_i$ does not make a significant difference, since the involutions
$(12)$ and $(23)$ are their own inverses. Multiplication by such an
involution occurs when we read an odd $a_i$, otherwise we copy the same
permutation and toggle the presence of a star. If $i$ is odd and $a_i$
is odd, then we follow an arrow leading from an even state to an odd
state and multiply by $(23)$. If $i$ is even and $a_i$ is odd, then we
follow an arrow from an odd state to an even state and multiply by
$(12)$. Note that the parity of $i$ may be determined simply by
checking which end of the arrow is marked by a star.

\begin{figure}[htb]
\begin{center}
  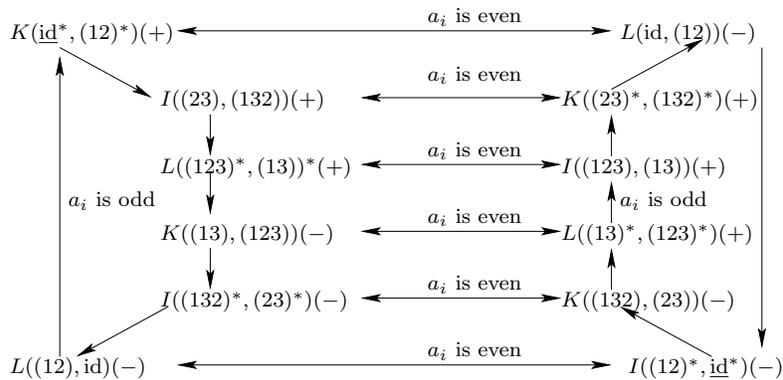
\end{center}
\caption{The automaton shown in Figure~\ref{fig:big_automaton} with
permutations labeling its states}
\label{fig:big_automaton_perm}
\end{figure}

Each state is labeled by a pair of permutations because there are two
valid starting states. The first permutation corresponds to starting
at the state labeled with the underlined identity permutation
$\underline{\mathrm{id}}^*$ in the upper left corner, while the second
permutation corresponds to starting at the state labeled with
$\underline{\mathrm{id}}^*$ in the lower right corner. We preserve the
symbols $K$, $I$, and $L$ indicating whether the current state is a
knot state, a link state, or an invalid stopping state.

Qazaqzeh, Yasein, and Abu-Qamar impose the additional restriction that
the diagram must be {\em canonical}, meaning that it contains an odd
number of twist boxes. In terms of our diagram this means that they
allow stopping only in an odd state. To obtain a canonical knot we must
either start in the upper left corner and stop at the state whose label
has first coordinate $(13)$ or $(132)$, or start in the lower right
corner and stop at the state whose label has second coordinate $(23)$
or $(123)$. Note that these permutations are pairwise distinct, so the
starting state can be reconstructed from the final permutation together
with the assumption that the diagram represents a canonical knot.

These possibilities are discussed in~\cite[Proposition~4.3]{Q}. Three
cases are distinguished there:

\begin{enumerate}
\item If the label of the final state is $(23)$ or $(123)$, then we must
have started in the lower right corner. This state carries a negative
sign, and so do both states reachable from it in a single step. The
writhe formula therefore begins with $-a_1-a_2$, to which we add the
contributions of the remaining partial denominators.

\item If the label of the final state is $(13)$ or $(132)$, then we must
have started in the upper left corner. This state has a positive sign,
but the sign of the next state depends on the parity of $a_1$. This
gives rise to two cases in~\cite[Proposition~4.3]{Q}: the writhe formula
begins with $a_1-a_2$ if $a_1$ is even and $a_1+a_2$ if $a_1$ is odd.
\end{enumerate}

A similar analysis of arriving at the even link states is left to the
reader. A proof of~\cite[Proposition~4.4]{Q} may be recovered in this
way. Here we only highlight the following observation made in~\cite{Q}.
}
For canonical diagrams (all $a_i>0$ and $n$ odd), relabeling the states
of Figure~\ref{fig:big_automaton} by the permutations used in~\cite{Q}
recovers their case distinctions.  The automaton packages those cases
into the following uniform recurrence.
In all cases a writhe formula of the form
\[
w(D)=\sum_{i=1}^n s_i a_i
\]
can be obtained with each $s_i\in\{1,-1\}$. The determination of
$s_1$ and $s_2$ varies from case to case, but for the remaining
coefficients the following rule holds in all cases:
\begin{equation}
\label{eq:signs}
s_i=
\begin{cases}
-s_{i-2} & \text{if $a_{i-1}$ is odd, $i-1$ is even, and $s_{i-1}=-1$,}\\
-s_{i-2} & \text{if $a_{i-1}$ is odd, $i-1$ is odd, and $s_{i-1}=1$,}\\
s_{i-2} & \text{otherwise.}
\end{cases}
\end{equation}

\vanish{Equation~\eqref{eq:signs} can be verified by inspecting
Figure~\ref{fig:s-automaton}. Consider the state reached immediately
before reading $a_{i-2}$. The sign associated with this state is
$s_{i-2}$, while $s_i$ is the sign of the state reached after reading
$a_{i-2}$ and $a_{i-1}$, that is, after moving along two consecutive
arrows. For each state there is exactly one pair of arrows that leads to
a state of opposite sign in two steps. These possibilities are listed in
Table~\ref{tab:signchange}.

\begin{table}[h]
\centering
\begin{tabular}{l|c|c||c|c}
State before $a_{i-2}$ & $a_{i-2}$ & $a_{i-1}$ & $i-1$ & $s_{i-1}$\\
\hline
$(1,3)(4,2)^*$/$(1,2)(4,3)^*$ & any  & odd & even & $-1$\\
$(1,2)(4,3)$/$(1,3)(4,2)$     & odd  & odd & odd  & $1$\\
$(2,1)(4,3)^*$/$(2,3)(4,1)^*$ & even & odd & even & $-1$\\
$(3,1)(4,2)$/$(3,2)(4,1)$     & any  & odd & odd  & $1$\\
$(3,2)(4,1)^*$/$(3,1)(4,2)^*$ & odd  & odd & even & $-1$\\
$(2,3)(4,1)$/$(2,1)(4,3)$     & any  & odd & odd  & $-1$\\
\end{tabular}
\vspace{0.1in}
\caption{A complete list of cases where $s_i=-s_{i-2}$}
\label{tab:signchange}
\end{table}

The choices are made in the first three columns; the last two columns
follow from these. In particular, $i-1$ is odd if and only if the state
reached after the first step is an even state, and the last column gives
the sign of the state reached after one step along the unique
sign-changing directed path of length two. Since the table represents a
complete enumeration of possibilities, equation~\eqref{eq:signs}
follows.}

Indeed,~\eqref{eq:signs} follows by inspecting the directed paths of
length two in Figure~\ref{fig:s-automaton}: the sign changes precisely
in the first two listed cases, and is otherwise unchanged.

Using our automaton also allows the computation of the writhe formula
for rational link diagrams containing an even number of twist boxes.
Rather than producing separate formulas case by case, we emphasize that
only the determination of $s_1$ and $s_2$ requires special
consideration; equation~\eqref{eq:signs} always determines the
remaining coefficients in the writhe formula.

\begin{example}{\em
Consider the unoriented rational link diagram encoded by the continued fraction
\[
[0,3,7,-5,3,-4,6,5,4,6,-8,9,3,2,4].
\]
For the orientation assignment corresponding to the initial state
$(2,1)(4,3)^*$, the final state determined by the automaton in
Figure~\ref{fig:big_automaton} is $I(1,2)(4,3)^*$. Hence this
orientation assignment is invalid. The other orientation assignment corresponds to the initial state
$(1,2)(4,3)^*$ and leads to the final state $K(2,1)(4,3)^*$. In this
case we obtain
$
s_1=s_2=-1$, $s_3=s_4=s_5=s_7=1$,
$s_8=-1$,
$s_9=1$,
$s_{10}=-1$,
$s_{11}=s_{12}=s_{13}=s_{14}=1$.
It follows that
\[
w(D)
=-3-7-5+3-4+6+5-4-6+8+9+3+2+4
=11.
\]

Note that this computation does not require drawing the link diagram,
and the procedure is readily programmable.
}
\end{example}

\section{Concluding remarks}
\label{sec:concl}

The Jones polynomial $V(t)$ of a link may be obtained from its HOMFLY
polynomial by setting $l=it^{-1}$ and
$m=i(t^{-\frac{1}{2}}-t^{\frac{1}{2}})$ according to the notation
of~\cite{Li-Mi}, or equivalently by setting
\begin{equation}
\label{eq:HOMFLYJones}
a=-t^{-1}\quad\mbox{and}\quad z=t^{-\frac{1}{2}}-t^{\frac{1}{2}},
\end{equation}
according to the notation in~\cite{DEH}. A formula for the HOMFLY
polynomial of an oriented rational link is given by Lickorish and
Millett~\cite{Li-Mi}. In their formula the rational link must be
represented by a diagram encoded by $[0,a_1,a_2,\ldots,a_n]$ in which
all partial denominators $a_i$ are even. A translation of the
Lickorish--Millett result to alternating links may be found in
\cite[Theorem~7.5]{DEH}. The Jones polynomial formulas obtained by
substituting into these HOMFLY polynomial formulas appear somewhat
cumbersome, although computationally they are about as ``efficient'' as
the Tait graph approach. Rather than reproducing the explicit
computations, we highlight the following observation.

The formulas in~\cite{DEH} contain substitutions into the {\em
Fibonacci polynomials $F_n(x)$}, defined by the initial conditions
\begin{equation}
F_0(x)=0\quad\mbox{and}\quad F_1(x)=1,
\end{equation}
and the recurrence
\begin{equation}
\label{eq:fibrec}
F_{n+1}(x)=xF_n(x)+F_{n-1}(x)\quad\mbox{for $n\geq1$}.
\end{equation}
These polynomials satisfy the following remarkable identity
\begin{equation}
\label{eq:fib}
F_n\!\left(x-\frac{1}{x}\right)=x^{\,n-1}\,[n]_{-x^{-2}}.
\end{equation}
A straightforward induction proof is left to the reader. As a
consequence of~\eqref{eq:fib}, the substitutions into the Fibonacci
polynomials in~\cite[Theorem~7.5]{DEH} take the form of $q$-analogues of
integers, similarly to the formulas obtained using the colored Tait
graph approach.

By Theorem~\ref{thm:t=-1}, the Jones polynomial encodes at least the
numerator of
the rational number representing the rational link. It is well known
that a rational link cannot be recovered from its Jones polynomial;
infinitely many counterexamples were constructed by
Kanenobu~\cite{Kan}, and a more systematic approach was developed by
Lawrence and Rosenstein~\cite{La-Ro}. The structure of our formulas,
which involve $q$-analogues of the partial denominators, suggests the
possible existence of another link invariant of comparable complexity
to the Jones polynomial which encodes the missing denominator of the rational number representing the rational link. These two invariants together would then uniquely
determine a rational link.

\end{document}